\documentclass[aop]{imsart}

\usepackage{amsthm,amsmath,amssymb}
\usepackage{graphicx}
\usepackage{mathptmx}

\RequirePackage[numbers]{natbib}
\RequirePackage[colorlinks,citecolor=blue,urlcolor=blue]{hyperref} 

\startlocaldefs
\numberwithin{equation}{section}
\theoremstyle{plain}
\newtheorem{thm}{Theorem}[section]
\newtheorem{lem}[thm]{Lemma}
\newtheorem{cor}[thm]{Corollary}
\newtheorem{prp}[thm]{Proposition}
\theoremstyle{remark}
\newtheorem{rmk}[thm]{Remark}
\theoremstyle{definition}

\endlocaldefs

\usepackage{xcolor}

\begin{document}
\newcommand{\Q}{\mathbb{Q}}
\newcommand{\R}{\mathbb{R}}
\newcommand{\I}{\mathbb{I}}
\newcommand{\Z}{\mathbb{Z}}
\newcommand{\N}{\mathbb{N}}
\newcommand{\F}{\mathcal{F}}
\newcommand{\p}{\mathbb{P}}
\newcommand{\M}{\mathcal{M}}        
\newcommand{\Pp}{\mathcal{P}}
\newcommand{\W}{\mathcal{W}}
\newcommand{\RR}{\mathcal{R}}
\newcommand{\E}{\mathbb{E}}
\newcommand{\h}{\mathcal{H}}
\newcommand{\Ii}{\mathrm{I}}
\newcommand{\eps}{\varepsilon}
\newcommand{\supp}{\mathrm{supp}}
\newcommand{\law}{\mathrm{Law}}
\newcommand{\leb}{\lambda}
\newcommand{\id}{\mathrm{id}}
\newcommand{\pr}{\mathrm{pr}}
\newcommand{\Ss}{\mathfrak{S}}
\newcommand{\inter}{\mathrm{int\,}}

\begin{frontmatter}

\title{Modified Massive  Arratia flow and\\ Wasserstein diffusion}
\runtitle{Massive Arratia flow}

\author{\fnms{Vitalii} \snm{Konarovskyi}\ead[label=e1]{konarovskiy@gmail.com}}
\and
\author{\fnms{Max-K.} \snm{von Renesse}\ead[label=e2]{renesse@math.uni-leipzig.de}}

\address{Max Planck Institut f\"{u}r Mathematik\\
in den Naturwissenschaften,\\
Inselstraße 22,\\
04103 Leipzig,\\
Germany\\
and\\
Department of Mathematics and\\
Informatics,\\
Yuriy Fedkovych Chernivtsi\\
National University,\\
Kotsubinsky Str. 2,\\
58012 Chernivtsi,\\
Ukraine\\
\printead{e1}}

\address{Universit\"{a}t Leipzig, \\
Fakult\"{a}t f\"{u}r Mathematik\\
und Informatik,\\
Augustusplatz 10,\\
04109 Leipzig,\\
Germany\\
\printead{e2}}

\affiliation{Max Planck Institut f\"{u}r Mathematik in den Naturwissenschaften and\\
Yuriy Fedkovych Chernivtsi National University,\\
Universit\"{a}t Leipzig}
\runauthor{V. Konarovskyi and M.-K. von Renesse}

\begin{abstract}
Extending previous work~\cite{Konarovskyi:2014:arx} by the first author we 
present  a variant of the Arratia flow, which consists of a collection of 
coalescing Brownian motions starting from every point of the unit interval. The 
important new feature of the model is that individual particles carry mass which 
aggregates upon coalescence and which scales the diffusivity of each particle in 
an inverse proportional way. In this work we relate the 
induced measure valued process to the Wasserstein diffusion 
of~\cite{Renesse:2009}. First, we present the process as  a martingale solution to a SPDE similar to~\cite{Renesse:2009}. Second, as our main result we show a 
Varadhan formula \cite{MR0208191} for short times  which is governed by 
the quadratic Wasserstein distance.
\end{abstract}



\end{frontmatter}

\section{Introduction and statement of main results} \subsection{Motivation}
Since its introduction in~\cite{Felix:2001} Otto's  formal infinite dimensional 
Riemannian calculus for optimal transportation has been the inspiration for 
numerous new results both in pure and applied
mathematics, see e.g.\ \cite{MR2842966,MR2480619,MR1760620,MR2237206,MR2142879}.
It can be considered a lift of conventional calculus of points to point ensembles resp.\ spatially continuous mass distributions.
It is therefore natural to ask whether this lifting procedure from points to
mass configurations has a probabilistic counterpart.
The fundamental object of such a theory  would
need to be an analogue of Brownian motion on the space of probability
measures  adapted to Otto's Riemannian structure of optimal
transportation. In~\cite{Renesse:2009} the second author together with
Sturm proposed a first
candidate
of such a measure valued Brownian motion (with drift), calling it
\textit{Wasserstein Diffusion}, and  showed among other things that its short
time asymptotics are
indeed governed by the geometry of optimal transport in the sense of a 
Varadhan
formula for short times governed by the  Wasserstein distance.
-- However, the
construction in~\cite{Renesse:2009} has several
limitations
since
it is strictly restricted to diffusing measures on the real line,  it
brings about additional seemingly  non-physical correction/renormalization
terms and lastly it is obtained by abstract Dirichlet form methods which e.g.\
do not allow for generic starting points of the evolution. Hence, in spite of
several
ad-hoc finite dimensional approximations~\cite{Renesse:2010,Stannat:2013,Sturm:2014} the process remained a
rather obscure object. Given the strong similarity of the SPDE representation 
of the Wasserstein Diffusion to the Dean-Kawasaki equation in physics
\cite{Dean:1996,Kawasaki199435} it is natural to ask for related measure  
valued diffusion processes which share  a similar multiplicative 
noise structure, giving rise to the same large deviation principles on short 
time scales. 

\smallskip

\subsection{Modified Massive Arratia Flow}
In this paper we give a different and very explicit construction of
another  diffusion process in the space of probability measures on the real line
which exhibits a similarity to the Wasserstein diffusion as discussed above. The
construction
is
based on a modification of the so-called Arratia flow of coalescing Brownian 
motions, which was introduced in~\cite{Arratia:1979} and which was later 
extensively studied by Dorogovtsev and 
coauthors~\cite{Dorogovtsev:2004,Dorogovtsev:2007:en,Dorogovtsev:2014,Ostapenko:2010,
Malovichko:2009} resp.\ Le Jan-Raimond~\cite{Le_Jan:2004}. As 
an important extension  the Brownian Web~\cite{MR2094432} has 
also received significant attention in recent studies.  

Our point of departure is another  
modification of the Arratia flow in~\cite{Konarovskiy:2010:TVP:en} by 
assigning a mass to each particle, which is aggregated when particles
coalesce and which controls the diffusivity of each particle 
in inverse
proportional way. In~\cite{Konarovskyi:2014:arx} it was shown 
for the first time that such a system can be
constructed starting with an infinitesimal mass particle at each point of the 
unit interval, i.e.\ such that the 
empirical measure of the particles almost surely converges in weak topology to  
the  uniform measure on the unit interval as time tends to zero. -- The 
resulting model, which we shall call \textit{modified
massive Arratia flow (MMAF)}, can best be described  in terms of a 
family of continuous martingales that describe the motion of the 
particles. Letting $D([0,1], C[0,T])$ denote
 the Skorokhod space of c\`{a}dl\'{a}g-functions from $[0,1]$ into the 
metric space $(C[0,T], d_\infty)$ of
continuous real valued trajectories over the time interval $[0,T]$  
 with the uniform distance $d_\infty$ and $\leb$ denote Lebesgue measure on 
$[0,1]$, the main
result of~\cite{Konarovskyi:2014:arx} reads as follows.
\begin{thm}\label{theorem_main_result}
There is a process $y \in D([0,1], C[0,T])$  such that
\begin{enumerate}
\item[(C1)] for all $u\in[0,1]$ the process $y(u,\cdot)$ is a continuous
square
integrable
martingale with respect to the filtration
\[\F_t=\sigma(y(u,s),\ u\in[0,1],\ s\leq t),\quad t\in[0,T]; \]
\item[(C2)] for all $u\in[0,1]$, $y(u,0)=u$;

\item[(C3)] for all $u<v$ from $[0,1]$ and $t\in[0,T]$, $y(u,t)\leq y(v,t)$;

\item[(C4)] for all $u,v\in[0,1]$ the joint quadratic variation of $y(u,\cdot)$ and $y(v,\cdot)$ is
$$
\left[y(u,\cdot),y(v,\cdot)\right]_t=\int_0^t\frac{\I_{\{\tau_{u,v}\leq
s\}}ds}{m(u,s)},
$$
where $m(u,t)=\leb\{v:\ \exists s\leq t\ \ y(v,s)=y(u,s)\}$, $\tau_{u,v}=\inf\{t:\
y(u,t)=y(v,t)\}\wedge T$.
\end{enumerate}
\end{thm}
Note that uniqueness in law of $y$ satisfying properties $(C1)-(C4)$ remains an 
important open problem. However, all subsequent results derived in 
this paper deal just with \textit{some} field of martingales 
satisfying properties $(C1)-(C4)$ above. In particular, uniqueness is not needed for any of our 
arguments.

We also point out that, in contrast to the classical Arratia flow,  the family 
of maps $\{y(\cdot,s)\}_{s\geq0}$ does not induce a (stochastic)  flow on the 
real  line, i.e.\ does not satisfy a cocycle property. Our terminology of a 
'modified massive Arratia flow' refers rather to the corresponding measure 
valued process\footnote{In fact $\mu_t$, $t\in[0,T]$, turns out to be a Markov process, but we will not stress this here.} 
$$
\mu_t:=y(\cdot,t)_\#\leb,\quad  t \in [0,T],
$$ 
which is obtained via the image (push forward)  of the uniform measure 
$\lambda$ on $[0,1]$ under the random maps $y(\cdot, t)$. The process $\mu_t$, $t\in[0,T]$, is the central object of our interest. In particular, 
$\mu_0 =\lambda$ in the present case, but our  arguments and constructions below 
can be modified to the case of more general starting measure, cf.~\cite{Konarovskyi:2017:EJP}.
For the sake of presentation, in the sequel we stick to the  $\mu_0=\lambda$ case.

For illustration and comparison to the standard Arratia flow we include 
here some numerical simulations. The red trajectory on the picture is the evolution of the center of mass of the particles which is a Brownian motion.

\begin{center}
\includegraphics[width=10cm]{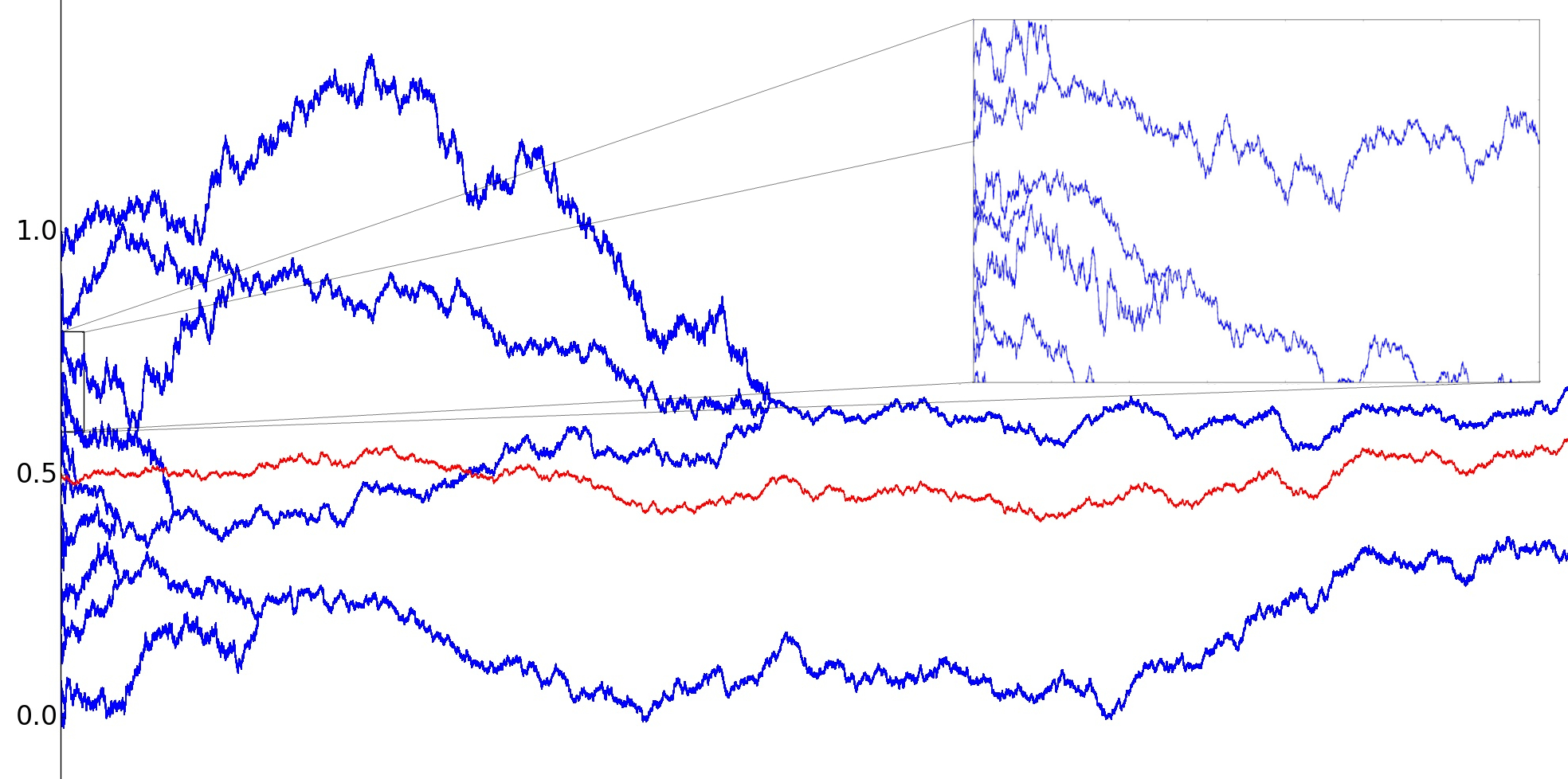}
\end{center}

\subsection{Main results for the Modfied Massive Arratia Flow}
\subsubsection{New construction and stochastic calculus for the MMAF}
The first result of this paper is a new simplified construction of a 
modified massive Arratia flow, using spatial discretization and a tightness 
argument. Second, we analyze the process $y(t)=y(\cdot,t)$, $t\in[0,T]$,
as an $L_2(\lambda)$-valued martingale and develop an associated stochastic calculus. We show  that for each $g \in L_2(\lambda)$, $s\mapsto
J(g)(s):= (g,y(s))_{L_2(\lambda)}$, where $(\cdot,\cdot)_{L_2(\lambda)}$ denotes the inner product in $L_2(\lambda)$, is a continuous square integrable martingale with quadratic variation process
$$
\left[J(g)\right]_t=\int_0^t\|\pr_{y(s)}g\|_{L_2(\lambda)}^2ds,\quad t\in[0,T].
$$
Here $\pr_{h} g$ is the orthogonal projection in $L_2(\lambda)$ of $g$ onto the subspace of $\sigma(h)$-measurable functions. This shows that the process $y(\cdot)$ is a martingale solution to the  infinite dimensional SDE
$$
dy(s) = \pr_{y(s)}  dW_s,
$$
where $W$ is cylindrical Brownian motion in the Hilbert space $L_2(\lambda)$.

By $(C3)$  the map $[0,1]\ni u \to y(u,t)$ is monotone (and c\`{a}dl\'{a}g), hence the
one-to-one map between  probability measures on
$\R$ and their quantile functions on $[0,1]$ yields an equivalent
parametrization of $y$ by the induced measure valued flow
\begin{equation}\label{f_def_mu_t}
\mu_t:=y(\cdot,t)_\#\leb,\quad  t \in [0,T],
\end{equation}
where $y(\cdot,t)_\#\leb(A)=\leb\{u:\ y(u,t)\in A\}$, $A \in 
\mathcal B (\R)$, denotes the image measure of $\lambda $ under the 
map $y(\cdot, t)$.\\

The process $\mu_t$, $t\in[0,T]$, and its relation to the
Wasserstein diffusion, is our main interest of this paper. Our first observation follows from the Ito formula for $y(t)$, $t\in[0,T]$, obtained 
in~\cite{Konarovskyi:2014:arx}.

\begin{prp}\label{prop_mart_prop} Let $\mu_t := y(\cdot,t)_\#\leb$. Then,
for each twice continuously differentiable function $f$ on $\R$ with bounded derivatives up to the second order
\[ M^f_t := \langle f ,
\mu_t\rangle - \int_0^t \langle f, \Gamma (\mu_s)
\rangle \,ds \]
is a continuous local martingale with quadratic variation process
\[ \left[ M^f\right]_t = \int_0^t \langle (f')^2,\mu_s\rangle ds,  \]
 where $\Gamma$ is defined as follows
\[  \langle f,  \Gamma (\nu)
\rangle = \frac{1}{2}\sum_{x \in \supp(\nu) } f''(x). \]
\end{prp}
We point out that  $\Gamma(\mu_t)$ is well defined since  
property $(P4)$  of section~\ref{subsection_properties_of_y} below implies, 
that $\supp(\mu_t)$ is a finite set for all $t\in(0,T]$ almost surely.

As a consequence of Proposition~\ref{prop_mart_prop}, $\mu_t$, $t\in[0,T]$, is a 
probability valued  martingale solution to
the SPDE
\begin{equation}
\label{modaflowspde} d\mu_t =  \Gamma(\mu_t)dt + \mathop{{\rm div}}( 
\sqrt {\mu_t} dW_t),
\end{equation}
which follows from a standard application of Ito's formula in finite 
dimensions. 

The SPDE~\eqref{modaflowspde} should be compared to the corresponding
SPDE for the Wasserstein
diffusion~\cite{Renesse:2010,Renesse:2009}\footnote{see 
also~\cite{MR2842966} for the connection to the Dean-Kawasaki equation 
(c.f.~\cite{Dean:1996}).} reading 
\[ d\mu_t =  \beta \Delta \mu_t dt +  \hat \Gamma(\mu_t)dt + \mathop{{\rm
div}}( \sqrt {\mu_t} dW_t),\]
with
\[ \langle f, \hat \Gamma(\nu)\rangle =
\sum\limits _{
I\in  \mathrm{gaps}(\nu)
} \left[ \frac{f''(I_+)+f''(I-)}2-\frac{f'(I_+)-f'(I_+)}{|I|}
\right] 
.\]
Thus, besides the apparent similarity of the second order part in the drift
operators $\Gamma$ and $\hat \Gamma$, both models share the same
singular multiplicative noise
which gives rise to the characteristic density $\langle (f')^2,
\mu\rangle$ in the quadratic variation process. Of course, this is 
the same
expression as the one appearing in Otto's definition \cite{Felix:2001}  of 
the Riemannian energy of an infinitesimal (tangential)   
perturbation  of a measure resp.\ in the 
Benamou-Brenier formula~\cite{MR1738163} for optimal
transportation. 

\subsubsection{Varadhan Formula for the short time asymptotics of the MMAF}

As the main achievement of the present paper we will make this 
connection more rigorous by showing that the small time fluctuations of the 
process are in fact governed, on an exponential scale, by the Wasserstein 
metric.  \\

To this aim, recall that the (quadratic) Wasserstein metric is defined as 
follows. For probability measures $\nu_1,\ \nu_2$ on the real line with finite 
second moments it is defined by
$$
d_{\mathcal W}(\nu_1,\nu_2)=\left(\inf_{\nu\in\chi(\nu_1,\nu_2)}\iint_{\R^2}|\xi-\eta|^2\nu(d\xi,d\eta)\right)^{\frac{1}{2}},
$$
where $\chi(\nu_1,\nu_2)$ denotes the set of all probability measures on $\R^2$ 
with marginals $\nu_1,\nu_2$. The main result of the present paper is the 
following version of the Varadhan formula for the measure valued diffusion 
$\mu_t$, $t\in[0,T]$. -- The precise conditions for a set $A \subset \Pp(\R)$ 
to be properly chosen for the statement are specified in section 
\ref{sububsection_prop_chosen} below.

\begin{thm}\label{thm_varadhan_formula}
Let $y$ satisfy $(C1)-(C4)$ and let $\mu_t$, $t\in[0,T]$, be defined 
by~\eqref{f_def_mu_t}. Then, for properly chosen sets $A \subset \Pp(\R)$
\begin{equation}\label{f_Varadhan_formula}
 \lim_{\eps \to 0} \eps \ln  \p\{\mu_{\eps}  \in A\} =
- \frac {\left(d_{\mathcal W}\left(\leb,A\right)\right)^2}{2},
\end{equation}
where the uniform distribution $\lambda$ on 
$[0,1]$ is considered as an element of $\Pp(\R)$.
\end{thm}

It should be noted that in Theorem~\ref{thm_varadhan_formula} we do not 
make any assumptions on the system $y$ other than $(C1)-(C4)$. Here, the particular construction leading to a system with these properties does not play any role. 

Theorem \ref{thm_varadhan_formula} is a large deviations
statement for the family of random measures $\mu_{\eps}$, involving the rate 
function 
\begin{equation*}
I(\eta) = \frac 1 2 \left(\inf_{\nu \in \chi(\lambda, \eta)} 
\iint_{\R^2} |\xi-\eta|^2 \nu(d\xi,d\eta) \right)^2= \frac 1 2 d_{\mathcal W}^2(\lambda, 
\eta).
\end{equation*}

We obtain it by contraction from a full large deviation principle for the 
family of 
processes $\{y^{\eps}(\cdot)\}_{\eps\in(0,1]}=\{y(\eps\cdot)\}_{\eps\in(0,1]}$. 
The latter is the main technical achievement of the present paper and it 
consumes the biggest part of it. 

\subsubsection{Large Deviation Principle for the MMAF}
For a precise  
 statement of the large deviation principle for the sequence 
$\{y^{\eps}(\cdot)\}_{\eps\in(0,1]}=\{y(\eps\cdot)\}_{\eps\in(0,1]}$
we need some notation. Let
$L_2(\rho)=L_2([0,1],\rho)$, $\rho(du)= \kappa(u)du$,  where
$\kappa:[0,1]\to[0,1]$
\begin{equation}\label{f_kappa}
\kappa(u)=\begin{cases}
              u^{\beta},\quad u\in[0,1/2],\\
              (1-u)^{\beta},\quad u\in(1/2,1],
          \end{cases}
\end{equation}
for some fixed $\beta>1$, and
$$
D^{\uparrow}=\{h\in D([0,1],\R):\ h\ \mbox{is non-decreasing}\}.
$$

Denote 
\begin{align}\label{f_set_H}
\begin{split}
\h=\{\varphi & \in C([0,T],L_2(\lambda)\cap D^{\uparrow}):\
\varphi(0)=\id\ \mbox{and}\\
&t\to \varphi(t) \in L_2(\lambda)\ \mbox{is absolutely continuous\footnotemark}\},
\end{split}
\end{align}
\footnotetext{A function $f(t)$, $t\in[0,T]$, taking values in a Hilbert space $H$ is called {\it
absolutely continuous} if there exists an integrable function $t\to h(t)\in H$ (in Bochner sense)
such that
$$
f(t)=f(0)+\int_0^th(s)ds,
$$
and we will denote the function $h$ by $\dot{f}$.}
$$
\Ii(\varphi)=\begin{cases}
              \frac{1}{2}\int_0^T\|\dot{\varphi}(t)\|^2_{L_2(\lambda)}dt,\quad
\varphi\in\h,\\
              +\infty,\quad \mbox{otherwise}.
           \end{cases}
$$

\begin{thm}\label{theorem_LDP}
The family of processes $\{y^{\eps}\}_{\eps\in(0,1]}$ satisfies a large
deviations principle in
the space
$C([0,T],L_2(\rho))$ with the good rate function $\Ii$, i.e. for any open set $G$ in $C([0,T],L_2(\rho))$
$$
\varliminf\limits_{\varepsilon\to 0}\varepsilon\ln\mathbb{P}\{y^{\eps}\in G\}\geq -\inf\limits_G \Ii
$$
and for any closed set $F$
$$
\varlimsup\limits_{\varepsilon\to 0}\varepsilon\ln\mathbb{P}\{y^{\eps}\in F\}\leq -\inf\limits_F \Ii.
$$
\end{thm}

Since the processes
$y^\eps(\cdot)$ solve
$$
dy^\eps(s) = \pr_{y^\eps(s)} \sqrt \eps dW_s,
$$
Theorem~\ref{theorem_LDP} appears as an instance of the classical
Freidlin-Wentzel LDP for solutions of SDE, but here we have to deal with
additional difficulties since the diffusion operator $g\to \sigma(g) = \pr_g$
is not continuous as an operator-valued map  on $L_2(\lambda)$,
and generally little is known about such
large deviation principles for solutions of a SDE with non-smooth coefficients
even in finite dimensions. In our case we can overcome these difficulties with
additional arguments, using the fact that $\sigma$ is continuous on strictly
monotone $y \in L_2(\lambda)$.

\subsubsection{Properly chosen subsets $A \subset 
\Pp(\R)$}\label{sububsection_prop_chosen}

In order to specify the conditions on the set $A$ for the validity of formula 
\eqref{f_Varadhan_formula},  let $\tau_\rho$ denote the image topology on
$\mathcal P(\R)$ of the $L_2(\rho)$-topology on $D^{\uparrow}([0,1])$ induced from
the bijection~
$$
\iota: \,  g \mapsto   g_\#\leb.
$$
We call  a set $A\subset \mathcal P 
(\R)$  displacement convex if it is the image of a convex subset of 
$D^{\uparrow}([0,1])$ under the map~$\iota$. A set is properly chosen for the 
validity of Varadhan's formula as in Theorem~\ref{thm_varadhan_formula}, for instance, if it is displacement convex $\tau_\rho$-closed with non-empty 
$\tau_{\rho}$-interior.

\begin{rmk}
It is possible to construct a process $y$ in a similar fashion on a circle $S$ with a proper notion of martingale on $S$. In this case the family $\{y^\eps(\cdot)=y(\eps\cdot)\}_{\eps\in(0,1]}$ will be exponentially tight in $C([0,T],L_2(\lambda))$, since the state space $L_2^{\uparrow}(\lambda)$ is compact. Consequently, the large deviation principle can be proved in $C([0,T],L_2(\lambda))$ and thus, it will imply that the Varadhan formula~\eqref{f_Varadhan_formula} holds for any measurable set $A$ that belongs to the space $\mathcal P(S)$ of probability measures on $S$ and
satisfies $\overline{\inter{A}}=\overline{A}$, for instance.
\end{rmk}

{\it The organization of the paper is as follows.} 
In section~\ref{section_constr} we give a
streamlined review of the construction of the modified massive Arratia flow 
from~\cite{Konarovskyi:2014:arx}\footnote{Here we construct the process directly on $[0,T]$ as a limit of particle systems, whereas in~\cite{Konarovskyi:2014:arx} the construction also included an $\eps\to 0$ limit for a sequence of processes on $[\eps,T]$.}. In section~\ref{section_SA} we introduce some 
elements of a stochastic
calculus relative to $y$ to the extent needed in the sequel. The final section~\ref{section_LDP}
is devoted to the proof of the large deviations principle Theorem~\ref{theorem_LDP}.

\section{Construction by a system of coalescing heavy diffusion particles}\label{section_constr}

\subsection{A finite number of particles}
We consider a finite system of particles which start from the points $\frac{k}{n}$, $k=1,\ldots,n$, with the mass $\frac{1}{n}$, where $n\in\N$ is fixed.

\begin{prp}\label{prop_finite_syst}
For each $n$, there exists a set of processes $\{x_k^n(t),\ k=1,\ldots,$ $n,\ t\in[0,T]\}$ that satisfies the
following conditions
\begin{enumerate}
\item[(F1)] for each $k$, $x_k^n$ is a continuous square integrable martingale with
respect to the filtration
$$
\F_t^n=\sigma(x_l^n(s),\ s\leq t,\ l=1,\ldots,n);
$$

\item[(F2)] for all $k$, $x_k^n(0)=\frac{k}{n}$;

\item[(F3)] for all $k<l$ and $t\in[0,T]$, $x_k^n(t)\leq x_l^n(t)$;

\item[(F4)] for all $k$ and $l$,
$$
[x_k^n,x_l^n]_t=\int_0^t\frac{\I_{\{\tau_{k,l}^n\leq s\}}ds}{m_k^n(s)},
$$
where $m_k^n(t)=\frac{1}{n}\#\{j:\ \exists s\leq t\ x_j^n(s)=x_k^n(s)\}$,
$\tau_{k,l}^n=\inf\{t:\ x_k^n(t)=x_l^n(t)\}\wedge T$ and $\#A$ denotes the number of points of $A$.
\end{enumerate}
\end{prp}

Such a system of processes can be constructed from a family of independent Wiener processes,
coalescing their trajectories. Moreover, $(F1)-(F4)$ uniquely determined the distribution of $x^n=(x_1^n,\ldots,x_n^n)$ in $(C[0,T])^n$ (see~\cite{Konarovskiy:2010:TVP:en}).

\subsection{Tightness of a finite system in the space $D([0,1],C[0,T])$}

Let 
$$
y_n(u,t)=\begin{cases}
              x_{\lfloor un\rfloor+1}^n(t),& u\in[0,1),\\
              x_n^n(t),& u=1,
           \end{cases}\quad t\in[0,T].
$$ 

\begin{prp}
The sequence $\{y_n(u,t),\ u\in[0,1],\ t\in[0,T]\}$ is tight in $D([0,1],C[0,T])$.
\end{prp}

The statement will follow from theorems~3.8.6 and~3.8.8~\cite{Ethier:1986} and Remark~3.8.9~ibid. The following lemmas~\ref{lemma_three_points},~\ref{lemma_first_points} and \ref{lemma_tightness_in_C} can be used to check conditions (8.39), (8.30) of~\cite{Ethier:1986} and (a) of Theorem~3.7.2~ibid., respectively.

\begin{lem}\label{lemma_three_points}
For all $n\in\N$, $u\in[0,2]$, $h\in[0,u]$ and $\lambda>0$
$$
\p\{d_{\infty}(y_n(u+h,\cdot),y_n(u,\cdot))>\lambda,\
d_{\infty}(y_n(u,\cdot),y_n(u-h,\cdot))>\lambda\}\leq\frac{4h^2}{\lambda^2}.
$$
Here $y_n(u,\cdot)=y_n(1,\cdot)$, $u\in[1,2]$, and $d_{\infty}$ is the uniform distance on
$[0,T]$.
\end{lem}

\begin{lem}\label{lemma_first_points}
For all $\beta>1$
$$
\lim_{\delta\to 0}\sup\limits_{n\geq 1}\E\left[d_{\infty}(y_n(\delta,\cdot),y_n(0,\cdot))^{\beta}\wedge
1\right]=0.
$$
\end{lem}

Lemmas~\ref{lemma_three_points} and~\ref{lemma_first_points} ware proved in~\cite{Konarovskyi:2014:arx} (see lemmas~2.2 and~2.3). The following statement is a new result.

\begin{lem}\label{lemma_tightness_in_C}
For all $u\in[0,1]$ the sequence $\{y_n(u,t),\ t\in[0,T]\}_{n\geq 1}$ is tight in
$C[0,T]$.
\end{lem}

\begin{proof}
To prove the lemma we use the Aldous tightness criterion (see e.g.
Theorem~3.6.5~\cite{Dawson:1993}), namely we show that
\begin{enumerate}
\item[(A1)] for all $t\in[0,T]$ the sequence $\{y_n(u,t)\}_{n\geq 1}$ is tight in $\R$;

\item[(A2)] for all $r>0$, each set of stopping times $\{\sigma_n\}_{n\geq 1}$ taking values in
$[0,T]$ and each sequence $\delta_n\searrow 0$
$$
\lim_{n\to\infty}\p\{|y_n(u,\sigma_n+\delta_n)-y_n(u,\sigma_n)|\geq r\}=0.
$$
\end{enumerate}

Note that $(A1)$ follows from Chebyshev's inequality and the estimate
\begin{align*}
\E |y_n(u,t)|&\leq\E\left|y_n(u,t)-\int_0^1y_n(q,t)dq\right|+\E\left|\int_0^1y_n(q,t)dq\right|\\
&\leq\E(y_n(1,t)-y_n(0,t))+\E\left|\int_0^1y_n(q,t)dq\right|\\
&=1+\E\left|\int_0^1y_n(q,t)dq\right|,
\end{align*}
where $\int_0^1y_n(q,t)dq$ is a Wiener process.

Condition $(A2)$ can be checked as follows. Similarly as in the proof of
Lemma~2.16~\cite{Konarovskyi:2014:arx}, we have that for each
$\alpha\in\left(0,\frac{3}{2}\right)$ there exists a constant $C$ such that for all $u\in[0,1]$ and
$n\geq 1$
$$
\E\frac{1}{m_n^{\alpha}(u,t)}\leq\frac{C}{\sqrt{t}},
$$
where
$$
m_n(u,t)=\begin{cases}
              m_{[un]+1}^n(t),& u\in[0,1),\\
              m_n^n(t),& u=1,
           \end{cases}\quad t\in[0,T].
$$
Thus, one can estimate
\begin{align*}
\varlimsup_{n\to\infty}\p\{|y_n(u,\sigma_n&+\delta_n)-y_n(u,\sigma_n)|\geq r\}\\
&\leq\frac{1}{r^2}\varlimsup_{n\to\infty}\E(y_n(u,\sigma_n+\delta_n)-y_n(u,\sigma_n))^2\\
&=\frac{1}{r^2}\varlimsup_{n\to\infty}\E\int_{\sigma_n}^{\sigma_n+\delta_n}\frac{1}{m_n(u,s)}ds\\
&=\frac{1}{r^2}\varlimsup_{n\to\infty}\E\int_0^T\I_{(\sigma_n,\sigma_n+\delta_n]}\frac{1}{m_n(u,
s)}ds\\
&\leq\frac{1}{r^2}\varlimsup_{n\to\infty}\left(\E\int_0^T\I_{(\sigma_n,\sigma_n+\delta_n]}
ds\right)^{\frac{1}{4}}\left(\E\int_0^T\frac{1}{m_n^{\frac{4}{3}}(u,s)}ds\right)^{\frac{3}{4}}\\
&\leq\frac{2^{\frac{3}{4}}CT^{\frac{3}{8}}}{r^2}\varlimsup_{n\to\infty}\delta_n^{\frac{1}{4}}=0.
\end{align*}
\end{proof}

\subsection{Martingale characterization of limit points (proof of Theorem~\ref{theorem_main_result})}

Since the space $D([0,1],C[0,T])$ is Polish, the tightness implies the relative compactness of $\{y_n(u,t),\ u\in[0,1],\ t\in[0,T]\}$ in $D([0,1],C[0,T])$. In this section we explain how one can prove that every limit point of $\{y_n\}$ satisfies $(C1)-(C4)$, which proves Theorem~\ref{theorem_main_result}. The idea is the same as in~\cite{Konarovskyi:2014:arx}.

Let $\{y_{n'}\}$ converge to $y$
weakly in the space $D([0,1],C[0,T])$ for some subsequence $\{n'\}$. By Skorokhod's theorem (see
Theorem~3.1.8~\cite{Ethier:1986}) we may suppose that $\{y_{n'}\}$ converge to $y$ in $D([0,1],C[0,T])$ a.s. For
convenience of notation we will suppose that the $\{y_n\}$ converge to $y$ in $D([0,1],C[0,T])$ a.s. Next, to prove the
theorem, first we show that $y_n(u,\cdot)$ tends to $y(u,\cdot)$ in $C[0,T]$ a.s. Note that, in
general, this does not follow from convergence in the space $D([0,1],C[0,T])$. So we need some
continuity property of $y(u,\cdot)$, $u\in[0,1]$, in $u$.

\begin{lem}\label{lemma_cont}
For all $u\in[0,1]$ one has
$$
\p\{y(u,\cdot)\neq y(u-,\cdot)\}=0.
$$
\end{lem}

\begin{proof}
The proof is similar to one of Lemma~2.9~\cite{Konarovskyi:2014:arx}.
\end{proof}

\begin{cor}\label{coroll_mart_prop}
For all $u\in[0,1]$
$$
y_n(u,\cdot)\to y(u,\cdot)\ \ \mbox{in}\ \ C[0,T]\ \ \mbox{a.s.}
$$
\end{cor}

Corollary~\ref{coroll_mart_prop} and Proposition~9.1.17~\cite{Jacod:2003} immediately imply
properties $(C1)-(C3)$. Property $(C4)$ can be proved by the following lemma and the representation
of $m(u,t)$ and $m_n(u,t)$ via $\tau_{u,v}$, $v\in[0,1]$, and $\tau_{u,v}^n$, $v\in[0,1]$, i.e.
\begin{align*}
m(u,t)&=\int_0^1\I_{\{\tau_{u,v}\leq t\}}dv,\\
m_n(u,t)&=\int_0^1\I_{\{\tau_{u,v}^n\leq t\}}dv,
\end{align*}
similarly as it was done in the proofs of lemmas~2.13 and 2.15~\cite{Konarovskyi:2014:arx}.

\begin{lem}\label{lemm_conv}
Let $\{z_n(t),\ t\in[0,T]\}_{n\geq 1}$, be a sequence of continuous local martingales (not necessary with respect to the same filtration) such that for all $n\geq 1$ and $s,t\in[0,\tau_n]$, $s<t$
\begin{equation}\label{rest_char}
[ z_n(\cdot)]_t-[ z_n(\cdot)]_s\geq p(t-s),
\end{equation}
where $\tau_n=\inf\{t:\ z_n(t)=0\}\wedge T$ and $p$ is a non-random positive constant. Let $z(t),\
t\in[0,T],$ be a continuous process such that
$$
z(\cdot\wedge\tau)=\lim_{n\to\infty}z_n(\cdot\wedge\tau_n)\ (\mbox{in}\ \ C([0,T],\R)) \ a.s.,
$$
where $\tau=\inf\{t:\ z(t)=0\}\wedge T$. Then
\begin{equation}\label{len_cen}
\tau=\lim_{n\to\infty}\tau_n\ \mbox{in probability}.
\end{equation}
\end{lem}

\begin{proof}
 The proof of this technical lemma can be found in~\cite[Lemma~2.10]{Konarovskyi:2014:arx}.
\end{proof}

\subsection{Some properties of the modified massive Arratia flow}\label{subsection_properties_of_y}

Let $y$ satisfy $(C1)-(C4)$. Then the following properties hold.
\begin{enumerate}
 \item[(P1)] For each $\alpha\in\left(0,\frac{3}{2}\right)$ the exists a constant $C$ such that for
all $u\in[0,1]$
$$
\E\frac{1}{m^{\alpha}(u,t)}\leq\frac{C}{\sqrt{t}},\quad t\in(0,T].
$$

 \item[(P2)] There exists a constant $C$ such that for all $u\in[0,1]$
$$
\E\int_0^t\frac{ds}{m(u,s)}\leq C\sqrt{t},\quad t\in[0,T].
$$

 \item[(P3)] There exists a constant $C$ such that for all $u\in[0,1]$
$$
\E (y(u,t)-u)^2\leq C\sqrt{t},\quad t\in[0,T].
$$

 \item[(P4)] Almost surely for all $t\in(0,T]$ the function $y(u,t),\ u\in[0,1]$, is a step function in
$D([0,1],\R)$ with a finite number of jumps. Moreover,
\begin{align}\label{f_coalescing_1}
\begin{split}
\p\{\forall u,v\in[0,1], t\in[0,T),\
y(u,t)&=y(v,t)\ \mbox{implies}\\ y(u,t+\cdot)&=y(v,t+\cdot)\}=1.
\end{split}
\end{align}
\end{enumerate}

\begin{rmk}\label{remark_coalescing}
According to~$(P4)$, hereafter we will suppose that for all $\omega\in\Omega$ and $t\in[0,T)$, $y(\cdot,t,\omega)$ is a step function in $D([0,1],\R)$ with a finite number of jumps. Also we assume that for all $u,v\in[0,1]$, $\omega\in\Omega$ and $t\in[0,T)$,  $y(u,t,\omega)=y(v,t,\omega)$ implies $y(u,t+\cdot,\omega)=y(v,t+\cdot,\omega)$.
\end{rmk}

Here $(P1)$ is the statement of Lemma~2.16~\cite{Konarovskyi:2014:arx}, $(P2)$ immediately follows from $(P1)$. Property $(P3)$ follows from $(P2)$ and $(C4)$. 

\begin{proof}[Proof of $(P4)$] 
We set
\begin{align*}
 \Omega'=&\{ \forall\ u,v\in[0,1]\cap\Q,\  t\in[0,T),\ y(u,t)=y(v,t)\\ 
 &\quad\mbox{implies}\ \ y(u,t+\cdot)=y(v,t+\cdot)\}\\
 \cap&\left\{\forall n\in\N\ \ \int_0^1\frac{du}{m(u,t_n)}<\infty,\ \mbox{where} \ t_n=\frac{1}{n}\wedge T\right\}.
\end{align*}
Since the set $[0,1]\cap\Q$ is countable, Proposition~2.3.4~\cite{Revuz:1999} and~$(P1)$ imply $\p\{\Omega'\}=1$.

Next we prove that 
\begin{equation}\label{f_coales_in_rat_case}
\begin{split}
&\mbox{for every}\ \omega\in\Omega',\ u\in[0,1]\cap\Q,\ v\in [0,u)\ \mbox{and}\ t\in[0,T)\\ 
&y(u,t,\omega)=y(v,t,\omega)\ \mbox{implies}\
y(u,t+\cdot,\omega)=y(v,t+\cdot,\omega).
\end{split}
\end{equation}
Indeed, if $y(u,t,\omega)=y(v,t,\omega)$, then by the monotonicity of $y(\cdot,t,\omega)$ (see $(C3)$), $y(u,t,\omega)=y(\widetilde{v},t,\omega)$ for all $\widetilde{v}\in[v,u)\cap\Q$. Hence $y(u,t+s,\omega)=y(\widetilde{v},t+s,\omega)$ for all $s\in[0,T-t]$. Using the right-continuity of $y(\cdot,t,\omega)$, we have $y(u,t+s,\omega)=y(v,t+s,\omega)$. This proves~\eqref{f_coales_in_rat_case}.

Let $\omega\in\Omega'$, $u\in[0,1]$, $v\in [0,u)$ and $t\in[0,T)$ be fixed and let $y(u,t,\omega)=y(v,t,\omega)$. If we show that there exists $\widetilde{u}\in[u,1]\cap\Q$ satisfying $y(\widetilde{u},t,\omega)=y(u,t,\omega)$, then~\eqref{f_coales_in_rat_case} will immediately imply~\eqref{f_coalescing_1}. To check this, we will use the fact that $\int_0^1\frac{d\widehat{u}}{m(\widehat{u},t_n,\omega)}$ is finite for all $n\in\N$. 

We fix some element $\widetilde{t}$ from $\{t_n,\ n\in\N\}$ such that $\widetilde{t}\leq t$ and assume that for all $\widetilde{u}\in(u,1]\cap\Q$\ \ $y(\widetilde{u},t,\omega)>y(u,t,\omega)$. Then the right-continuity of $y(\cdot,t,\omega)$ and its monotonicity imply that there exists a sequence $\{u_n\}_{n\geq 1}$ strongly decreasing to $u$ such that $y(u_{n+1},t,\omega)<y(u_n,t,\omega)$ for all $n\in\N$. Next, we set 
$$
\widetilde{u}_n=\inf\{u':\ y(u',t,\omega)=y(u_n,t,\omega)\},\quad n\in\N.
$$
Since $y(\cdot,t,\omega)$ is right-continuous, we have $y(\widetilde{u}_n,t,\omega)=y(u_n,t,\omega)$. Moreover, $\{\widetilde{u}_n\}_{n\geq 1}$ also strongly decreases to $u$ and $y(\widetilde{u}_{n+1},t,\omega)<y(\widetilde{u}_n,t,\omega)$ for all $n\in\N$. Consequently, for all $u'\in (\widetilde{u}_{n+1},\widetilde{u}_n)\cap\Q$ and $u''\in (\widetilde{u}_{n+2},\widetilde{u}_{n+1})\cap\Q$, $n\in\N$,  \ $y(u'',t,\omega)<y(u',t,\omega)$, by the monotonicity of $y(\cdot,t,\omega)$ and the choice of the sequence $\{\widetilde{u}_n\}_{n\geq 1}$. Thus, $y(u'',r,\omega)<y(u',r,\omega)$ also for each $r\in[0,t]$, since $u',\ u''$ are rational and $\omega$ was taken from $\Omega'$. 
Now we can estimate for every $\widehat{u}\in(\widetilde{u}_{n+1},\widetilde{u}_n)$, $n\in\N$,
$$
m(\widehat{u},\widetilde{t},\omega)=\leb\{u':\ \exists r\leq\widetilde{t} \ \ y(u',r,\omega)=y(\widehat{u},r,\omega)\}\leq \widetilde{u}_n-\widetilde{u}_{n+1}.
$$
So,
$$
\int_0^1\frac{d\widehat{u}}{m(\widehat{u},\widetilde{t},\omega)}\geq \sum_{n=1}^{\infty}\int_{\widetilde{u}_{n+1}}^{\widetilde{u}_n}\frac{d\widehat{u}}{m(\widehat{u},\widetilde{t},\omega)}\geq \sum_{n=1}^{\infty}\int_{\widetilde{u}_{n+1}}^{\widetilde{u}_n}\frac{d\widehat{u}}{\widetilde{u}_n-\widetilde{u}_{n+1}}=+\infty.
$$
But this contradicts the finiteness of the integral $\int_0^1\frac{d\widehat{u}}{m(\widehat{u},\widetilde{t},\omega)}$.
Consequently \eqref{f_coalescing_1} holds.

Next, let $t\in(0,T]$ be fixed. We are going to show that $y(\cdot,t)$ is a step function with a finite number of jumps a.s. Let $N(t)$ be a number of distinct points of $B_t=\{y(u,t),\ u\in[0,1]\}$ (that can be equal $+\infty$, if $B_t$ has infinitely many points). Then under~\eqref{f_coalescing_1} one can see that
\begin{equation}\label{f_N_def}
N(t)=\int_0^1\frac{du}{m(u,t)}\quad\mbox{a.s.}
\end{equation}
Indeed, let for fixed $\omega$, that we omit in the notation, $\pi(u,t)=\{v:\ y(v,t)=y(u,t)\}$, $u\in[0,1]$. Then by~\eqref{f_coalescing_1}, we have $m(u,t)=\leb(\pi(u,t))$. Consequently, \eqref{f_N_def} holds, if $N(t)$ is finite. Next, we suppose that $N(t)=+\infty$ and set $A_t=\{u:\ m(u,t)>0\}$. Note that $\int_0^1\frac{du}{m(u,t)}=+\infty$ is enough to check only for the case $\leb(A_t)=1$. So, assuming that $\leb(A_t)=1$ and using the fact that the number of distinct points of $B_t$ is infinite and $y(\cdot, t)$ is non-decreasing, it is easily seen that there exists a set of $\{u_k,\ k\in\N\}\subset [0,1]$ such that $y(u_k,t)\neq y(u_l,t)$ for all $k\neq l$ and $m(u_k,t)>0$, $k\in\N$. Now, we can estimate
$$
\int_0^1\frac{du}{m(u,t)}\geq\sum_{k=1}^n\int_{\pi(u_k)}\frac{du}{m(u,t)}=
\sum_{k=1}^n\int_{\pi(u_k)}\frac{du}{\leb(\pi(u_k))}=n.
$$
Letting $n\to\infty$, we get~\eqref{f_N_def}.

Thus, $N(t)$ must be finite a.s., by $(P1)$.

Also we would like to note here that~\eqref{f_coalescing_1} yields that almost surely for all $t\in(0,T]$ \ $y(\cdot,t)$ is a step function with a finite number of jumps.
\end{proof}

\section{Some elements of stochastic analysis for the system of heavy diffusion particles}\label{section_SA}

In this section $L_2$ will denote the space of square integrable measurable functions on $[0,1]$ with respect to Lebesgue
measure and $\|\cdot\|_{L_2}$ the usual norm in $L_2$.

\subsection{Definition of a stochastic integral for predictable $L_2$-valued processes}\label{stochastic_integral}

In this section we give a self-contained construction of the stochastic integral with respect to $y$ with emphasis on a simpler class of integrands then e.g. in Krylov-Rozovskii~\cite{Krylov:1981}.

As before let $\pr_ba$ denote the projection of $a$ onto the space of $\sigma(b)$-measurable functions from $L_2$.

\begin{lem}\label{lemma_y_is_martingale}
 For each $a\in L_2$ the process $(y(t),a)$, $t\in[0,T]$, is a continuous square integrable $(\F_t)$-martingale with the quadratic variation
 $$
 [(y(\cdot),a)]_t=\int_0^t\|\pr_{y(s)}a\|_{L_2}^2ds.
 $$
\end{lem}

\begin{proof}
 First note that $M(t):=(y(t),a)$, $t\in[0,T]$, is a continuous square integrable $(\F_t)$-martingale, since for each $u\in[0,1]$, $y(u,t)$, $t\in[0,T]$, is. Hence, it is enough to check that for all $0\leq s< t\leq T$
 $$
 \E\left[\left.(M(t)-M(s))^2\right|\F_s\right]=\E\left[\left.\int_s^t\|\pr_{y(r)}a\|_{L_2}^2dr\right|\F_s\right].
 $$
 Since for each $u,v\in[0,1]$ the joint quadratic variation of $y(u,\cdot)$ and $y(v,\cdot)$ equals $\int_s^t\frac{\I_{\{\tau_{u,v}\leq r\}}}{m(u,r)}dr$, we have
 \begin{align*}
 \E&\left[\left.(M(t)-M(s))^2\right|\F_s\right]\\
 &=\E\left(\left.\int_0^1\int_0^1a(u)a(v)(y(u,t)-y(u,s))(y(v,t)-y(v,s))dudv\right|\F_s\right)\\
 &=\E\left(\left.\int_0^1\int_0^1a(u)a(v)\left[\int_s^t\frac{\I_{\{\tau_{u,v}\leq r\}}}{m(u,r)}dr\right]dudv\right|\F_s\right)\\
 &=\E\left(\left.\int_s^t\left[\int_0^1\int_0^1a(u)a(v)\frac{\I_{\{\tau_{u,v}\leq r\}}}{m(u,r)}dudv\right]dr\right|\F_s\right).
 \end{align*}
 
 By Fubini's theorem, we obtain
 \begin{equation}\label{f_calcul_of_variation}
 \begin{split}
 \int_0^1\int_0^1&a(u)a(v)\frac{\I_{\{\tau_{u,v}\leq r\}}}{m(u,r)}dudv\\
 &=\int_0^1\frac{a(u)}{m(u,r)}\left(\int_0^1a(v)\I_{\{\tau_{u,v}\leq r\}}dv\right)du\\
 &=\int_0^1\frac{a(u)}{m(u,r)}\left(\int_{\pi(u,r)}a(v)dv\right)du,
 \end{split}
 \end{equation}
 where $\pi(u,t)=\{v:\ y(v,t)=y(u,t)\}$. Here we have used the equality $\pi(u,t)=\{v:\ \tau_{u,v}\leq r\}$, which follows from Remark~\ref{remark_coalescing}. We note that for each $\omega$ and $r$ the operator $\pr_{y(r,\omega)}$ is a usual projection (in $L_2(\lambda)$)  onto the subspace of all $\sigma(y(r,\omega))$-measurable functions and moreover
 $$
 \left(\pr_{y(r,\omega)}a\right)(u)=\frac{1}{m(u,r,\omega)}\int_{\pi(u,r,\omega)}a(v)dv
 $$
 because $y(r,\omega)$ is a step function according to Remark~\ref{remark_coalescing}. Consequently,
 the left hand side of~\eqref{f_calcul_of_variation} equals
 $$
 \int_0^1a(u)\left(\pr_{y(r)}a\right)(u)du=\|\pr_{y(r)}a\|_{L_2}^2.
 $$
 The lemma is proved.
\end{proof}

 By the polarization equality, the following corollary holds.

\begin{cor}\label{coroll_joint_variation}
 For each $a,b\in L_2$ we have
 $$
 [(y(\cdot),a),(y(\cdot),b)]_t=\int_0^t(\pr_{y(s)}a,\pr_{y(s)}b)ds.
 $$
\end{cor}

Let $\{e_n\}_{n\geq 1}$ be a fixed orthonormal basis in $L_2$, and $f(t)$, $t\in[0,T]$, be a predictable process taking values in $L_2$ with
\begin{equation}\label{f_norm_of_A2}
\E\int_0^T\|f(t)\|_{L_2}^2dt<\infty.
\end{equation}

We define the integral of $f$ with respect to $y$ as the series 
\begin{equation}\label{f_integral}
\int_0^t(f(s),dy(s))=\int_0^t\int_0^1f(u,s)dy(u,s)du:=\sum_{n=1}^{\infty}\int_0^t(f(s),e_n)d(y(s),e_n),
\end{equation}
which converges in $\M_2$ according to the following proposition, where $\M_2$ denotes the space of real valued continuous square integrable $(\F_t)$-martingales $M(t)$, $t\in[0,T]$, with the norm
$$
\|M\|_{\M_2}=\left(\E M^2(T)\right)^{\frac{1}{2}}.
$$

\begin{prp}
 The series~\eqref{f_integral} converges in $\M_2$ and is a continuous square integrable $(\F_t)$-martingale with the quadratic variation
 \begin{equation}\label{f_mart_characteristic}
 \left[\int_0^{\cdot}(f(s),dy(s))\right]_t=\int_0^t\|\pr_{y(s)} f(s)\|_{L_2}^2ds.
 \end{equation}
\end{prp}

\begin{proof}

 We set for each $n\in\N$
 $$
 S_n(t)=\sum_{k=1}^n\int_0^t(f(s),e_k)d(y(s),e_k),\quad t\in[0,T].
 $$
 Corollary~\ref{coroll_joint_variation} and a simple calculation yield that $S_n$ belongs to $\M_2$ and has the quadratic variation
 $$
 [S_n]_t=\int_0^t\left\|\pr_{y(s)}\sum_{k=1}^nf_k(s)e_k\right\|_{L_2}^2ds.
 $$
 Moreover, for each $1\leq n<p$
 \begin{align*}
 \|S_p-S_n\|_{\M_2}^2=\E(S_p(T)-S_n(T))^2\leq \E\int_0^T\left\|\sum_{k=n+1}^pf_k(t)e_k\right\|_{L_2}^2dt.
 \end{align*}
 By the dominated convergence theorem and assumption~\eqref{f_norm_of_A2}, $\|S_p-S_n\|_{\M_2}\to 0$ as $n,p\to\infty$. Thus, the sequence $\{S_n\}_{n\geq 1}$ converges in $\M_2$, by the completeness of $\M_2$. Next,~\eqref{f_mart_characteristic} follows from Lemma~B.11~\cite{Engelbert:2005}. The proposition is proved.
\end{proof}

\begin{rmk}
 Let $f(t)$, $t\in[0,T]$, be a predictable $L_2$-valued process such that
\begin{equation}\label{f_set_A_of_predictable_functions}
\int_0^T\|f(t)\|_{L_2}^2dt<\infty\quad\mbox{a.s.}
\end{equation}
Then using a localization sequence of stopping times, one can define the stochastic integral $\int_0^{\cdot}(f(s),dy(s))$ that is a continuous local square integrable $(\F_t)$-martingale with the quadratic variation given by~\eqref{f_mart_characteristic}.
\end{rmk}

\subsection{Girsanov's theorem}

In this section we construct a system of coalescing diffusion particles with drift that will be needed in Section~\ref{lower_bound} for the proof of the lower bound in LDP. So, fix
 a predictable $L_2$-valued process $\varphi$ satisfying~\eqref{f_set_A_of_predictable_functions} and consider on $(\Omega,\F)$ the new measure
$$
\p^{\varphi}(A)=\E\I_A\exp\left\{\int_0^T(\varphi(s),dy(s))-\frac{1}{2}\int_0^T\|\pr_{y(t)}
\varphi(t)\|_{L_2}^2dt\right\},\quad A\in\F.
$$
If
\begin{equation}\label{f_min_1_cond}
\E \exp\left\{\int_0^T(\varphi(s),dy(s))-\frac{1}{2}\int_0^T\|\pr_{y(t)} \varphi(t)\|_{L_2}^2dt\right\}=1,
\end{equation}
then $\p^{\varphi}$ is a probability measure.

\begin{thm}\label{theorem_Girsanov}
Let $\varphi$ satisfy~\eqref{f_min_1_cond}. Then the random element $\{y(u,t),\ u\in[0,1],\ t\in[0,T]\}$ in $D([0,1],C[0,T])$ satisfies the following
properties under $\p^{\varphi}$
\begin{enumerate}
\item[(D1)] for all $u\in[0,1]$ the process
$$
\eta(u,\cdot)=y(u,\cdot)-\int_0^{\cdot}\left(\pr_{y(s)}\varphi(s)\right)(u)ds
$$
is a continuous local square integrable $(\F_t)$-martingale;

\item[(D2)] for all $u\in[0,1]$, $y(u,0)=u$;

\item[(D3)] for all $u<v$ from $[0,1]$ and $t\in[0,T]$, $y(u,t)\leq y(v,t)$;

\item[(D4)]  for all $u,v\in[0,1]$ and $t\in[0,T]$,
$$
[\eta(u,\cdot),\eta(v,\cdot)]_t=\int_0^t\frac{\I_{\{\tau_{u,v}\leq s\}}ds}{m(u,s)}.
$$
\end{enumerate}
\end{thm}

Note that $(D2)$ and $(D3)$ immediately follows from the absolute continuity of $\p^{\varphi}$. To prove $(D1)$ and $(D4)$ we state an auxiliary lemma.

\begin{lem}\label{lemma_representation_of_y(u)}
For each $u\in[0,1]$
$$
y(u,t)=u+\int_0^t\int_0^1\frac{\I_{\pi(u,s-)}(q)}{m(u,s-)}dy(q,s)dq.
$$
\end{lem}

\begin{proof}
Setting $f(q,s)=\frac{\I_{\pi(u,s-)}(q)}{m(u,s-)}$ and using $(P2)$, we
have
\begin{align*}
\E\int_0^T\|f(s)\|_{L_2}^2ds&=\E\int_0^T\int_0^1\frac{\I_{\pi(u,s-)}(q)}{m^2(u,s-)}dsdq\\
&=\E\int_0^T\left(\frac{1}{m^2(u,s)}\int_{\pi(u,s)}dq\right)ds\\
&=\E\int_0^T\frac{1}{m(u,s)}ds<\infty.
\end{align*}

Next, put
\begin{align*}
\sigma_0&=t,\\
\sigma_k&=\inf\{s:\ N(s)\leq k\}\wedge t,\quad k\in\N,
\end{align*}
where $N(t)=\int_0^1\frac{1}{m(u,t)}$, $t\in[0,T]$, denotes a number of distinct points in $\{y(u,t),\ u\in[0,1]\}$ and is
an $(\F_t)$-adapted c\`{a}dl\'{a}g process, and note that $\sigma_k$ is an $(\F_t)$-stopping time, $\sigma_k\geq\sigma_{k+1}$. So,
\begin{align*}
\int_0^t\int_0^1 &f(q,s)dy(q,s)dq=\sum_{k=0}^{\infty}\int_{\sigma_{k+1}}^{\sigma_k}\int_0^1
\frac{\I_{\pi(u,s-)}(q)}{m(u,s-)}dy(q,s)dq\\
&=\sum_{k=0}^{\infty}\int_0^1
\frac{\I_{\pi(u,\sigma_{k+1})}(q)}{m(u,\sigma_{k+1})}(y(q,\sigma_k)-y(q,\sigma_{k+1}))dq\\
&=\sum_{k=0}^{\infty}\int_{\pi(u,\sigma_{k+1})}dq\frac{1}{m(u,\sigma_{k+1})}(y(u,\sigma_k)-y(u,
\sigma_{k+1}))\\
&=\sum_{k=0}^{\infty}(y(u,\sigma_k)-y(u,\sigma_{k+1}))=y(u,t)-u.
\end{align*}
The lemma is proved.
\end{proof}

\begin{cor}\label{corollary_characteristic_f_y}
For each predictable $L_2$-valued process $f$ satisfying~\eqref{f_set_A_of_predictable_functions} and $u\in[0,1]$
$$
\left[\int_0^{\cdot}(f(s),dy(s)),y(u,\cdot)\right]_t=\int_0^t\left(\pr_{y(s)}
f(s)\right)(u)ds,\quad t\in[0,T].
$$
\end{cor}

\begin{proof}[Proof of Theorem~\ref{theorem_Girsanov}]
The proof of the assertion follows from Girsanov's theorem (see
Theorem~5.4.1~\cite{Watanabe:1981:en}) and Corollary~\ref{corollary_characteristic_f_y}.
\end{proof}

\begin{rmk}
For predictable $L_2$-valued functions $f$ satisfying~\eqref{f_norm_of_A2} (resp.~\eqref{f_set_A_of_predictable_functions}) we can construct the stochastic integral with respect to
the flow $\{y(u,t),\ u\in[0,1],\ t\in[0,T]\}$ satisfying conditions $(D1)-(D4)$ in the same way as
in the case of conditions $(C1)-(C4)$. Moreover,
$$
\int_0^t(f(s),dy(s))=\int_0^t(f(s),\pr_{y(s)}\varphi(s))ds+\int_0^t(f(s),d\eta(s))
$$
and $\int_0^{\cdot}(f(s),d\eta(s))$ is a continuous square integrable (resp. local square
integrable) $(\F_t)$-martingale with the quadratic variation
$$
\left[\int_0^{\cdot}(f(s),d\eta(s))\right]_t=\int_0^t\|\pr_{y(s)} f(s)\|_{L_2}^2ds.
$$
\end{rmk}

\section{Large deviation principle for the modified massive Arratia flow}\label{section_LDP}

\subsection{Exponential tightness}

In this section we prove exponential tightness of the modified massive Arratia flow. In order to prove this we will use ``exponentially fast'' version of Jakubowski's
tightness criterion (see Theorem~A.1~\cite{MR1607396}).
So, let $\{y(u,t), u\in[0,1], t\in[0,T]\}$
be a random element in $D([0,1],C[0,T])$ satisfying $(C1)-(C4)$ and $\rho(du)=\kappa(u)du$, where $\kappa$ is given by~\eqref{f_kappa}.

Since $y(\cdot,t)$, $t\in[0,T]$, is a continuous $L_2(\rho)$-valued process, we will establish exponential tightness of
$\{y^{\eps}\}_{\eps\in(0,1]}$ in the space $C([0,T],L_2(\rho))$, where $y^{\eps}(t)=y(\cdot,\eps t),\
t\in[0,T]$.

By Theorem~A.1~\cite{MR1607396}, $\{y^{\eps}\}_{\eps\in(0,1]}$ is exponential tight in
$C([0,T],L_2(\rho))$, i.e. for every $M>0$ there exists a compact $K_M\subset C([0,T],L_2(\rho))$
such that
 $$
 \varlimsup_{\eps\to 0}\eps\ln\p\{y^{\eps}\notin K_M\}\leq -M,
 $$
if and only if
\begin{enumerate}
 \item[(E1)] for every $M>0$ there exists a compact $K_M\subset L_2(\rho)$ such that
\begin{equation}\label{f_tightness_comp_in_L}
 \varlimsup_{\eps\to 0}\eps\ln\p\{\exists t\in[0,T]:\ y^{\eps}(t)\notin K_M\}\leq -M;
\end{equation}

 \item[(E2)] for every $h\in L_2(\rho)$ the sequence $\{(h,y^{\eps}(\cdot))_{L_2(\rho)}\}_{\eps\in(0,1]}$ is
exponentially tight in $C([0,T],\R)$, where $(\cdot,\cdot)_{L_2(\rho)}$ denotes the inner product in
$L_2(\rho)$.
\end{enumerate}

Since $y(u,t),\ u\in[0,1]$, is non-decreasing for all $t\in[0,T]$, to find the compact $K_M\subset
L_2(\rho)$ satisfying~\eqref{f_tightness_comp_in_L} it suffices to control the behavior of processes
$y(u,t),\ t\in[0,T]$, for $u$ close to 0 or 1. Note that the diffusion rate of the process
$y(u,t),\ t\in[0,T]$, tends to infinity as $t\to 0$. But
\begin{equation}\label{f_estim_of_trajectories}
 M_*(u,t)\leq y(u,t)\leq M^*(u,t),\quad t\in[0,T],
\end{equation}
where
\begin{align*}
 M^*(u,t)&=\frac{1}{1-u}\int_u^1y(v,t)dv,\\
 M_*(u,t)&=\frac{1}{u}\int_0^uy(v,t)dv
\end{align*}
and
\begin{align}\label{f_charakteristics_of_M}
\frac{d[ M^*(u,\cdot)]_t}{dt}\leq \frac{1}{1-u},\quad
\frac{d[ M_*(u,\cdot)]_t}{dt}\leq \frac{1}{u}.
\end{align}
The latter inequalities follow from the simple relations
\begin{align*}
 M^*(u,t)&=\frac{1}{1-u}\int_0^t (\I_{[u,1]},dy(s)),\\
 M_*(u,t)&=\frac{1}{u}\int_0^t (\I_{[0,u]},dy(s)),
\end{align*}
where these sort of integrals ware defined in Section~\ref{stochastic_integral},
and the formula for the quadratic variation of the stochastic integral~\eqref{f_mart_characteristic}. Indeed,
\begin{align*}
 \frac{d[ M^*(u,\cdot)]_t}{dt}=\frac{1}{(1-u)^2}\|\pr_{y(t)}\I_{[u,1]}\|_{L_2}^2\leq\frac{1}{(1-u)^2}\|\I_{[u,1]}\|_{L_2}^2=\frac{1}{1-u}.
\end{align*}
The inequality $\frac{d[ M_*(u,\cdot)]_t}{dt}\leq \frac{1}{u}$ can be obtained in the same way.

In Lemma~\ref{lemma_E1} we will use inequalities~\eqref{f_estim_of_trajectories} in order to find a compact $K_M\subset L_2(\rho)$ for each $M>0$ such that~\eqref{f_tightness_comp_in_L} holds.

Since $y(u,t),\ u\in[0,1]$, belongs to $D([0,1],\R)$ and is a non-decreasing function, we will often work
with non-decreasing functions
$$
D^{\uparrow}=\{h\in D([0,1],\R):\ h\ \mbox{is non-decreasing}\}.
$$

\begin{lem}\label{lemma_compact_in_L_2}
 The set
$$
A_M=\left\{h\in D^{\uparrow}:\ h(1/n)\geq -Mn\ \mbox{and}\ h(1-1/n)\leq Mn,\ n\in\N\right\}.
$$
is compact in $L_2(\rho)$ for all positive $M$.
\end{lem}

\begin{proof}
First we prove that $A_M\subset L_2(\rho)$. Let $h\in A_M$. Without loss of generality let
$h$ be positive on $[1/2,1]$ and negative on $[0,1/2]$. Then
\begin{align*}
 \int_0^1h^2(u)\rho(du)&=\int_0^1h^2(u)\kappa(u)du\leq
C\sum_{n=2}^{\infty}\frac{M^2n^2}{n^{\beta}}\left(\frac{1}{n-1}-\frac{1}{n}\right)<C_1
\end{align*}
and $C_1$ is independent of $h$.

Next, take a sequence $\{h_k\}_{k\geq 1}$ in $A_M$. Since $\{h_k\}_{k\geq 1}\subset D^{\uparrow}$,
there exists a subsequence $\{h_{k'}\}$ that convergences to $h\in D^{\uparrow}$ for all $u\in[0,1]$ except possibly countably many points, and hence
also $\rho$-a.e. Since
$|h_k(u)|\leq f(u),\ u\in[0,1]$, where
$$
f(u)=\begin{cases}
       Mn,\quad u\in[1-1/(n-1),1-1/n),\\
       Mn,\quad u\in[1/n,1/(n-1)),
      \end{cases}
$$
and $f\in L_2(\rho)$, $\|h_{k'}\|_{ L_2(\rho)}\to\|h\|_{ L_2(\rho)}$, by the dominated convergence
theorem. Consequently, this and Lemma~1.32~\cite{Kallenberg:2002} imply $h_{k'}\to h$ in
$L_2(\rho)$. The lemma is proved.
\end{proof}

\begin{lem}\label{lemma_E1}
The family of processes $\{y^{\eps}\}_{\eps\in(0,1]}$ satisfies $(E1)$, i.e. for every $M>0$ there
exists a compact $K_M\subset L_2(\rho)$, such that~\eqref{f_tightness_comp_in_L} holds.
\end{lem}

\begin{proof}
 By Lemma~\ref{lemma_compact_in_L_2}, we can take $K_M=A_L$ and show that for some
$L>0$~\eqref{f_tightness_comp_in_L} holds. So,
\begin{align*}
 \p\{\exists t\in[0,T]:\ y^{\eps}(t)\notin A_L\}&\leq\sum_{n=1}^{\infty}\p\{\exists t\in[0,T]:\
y^{\eps}(1/n,t)< -Ln\}\\
&+\sum_{n=1}^{\infty}\p\{\exists t\in[0,T]:\ y^{\eps}(1-1/n,t)> Ln\}.
\end{align*}

Using~\eqref{f_estim_of_trajectories} and \eqref{f_charakteristics_of_M}, we estimate for fixed $n\in\N$
\begin{align*}
\p\{\exists t\in[0,T]&:\ y(1-1/n,\eps t)> Ln\}=\p\left\{\sup_{t\in[0,T]}y(1-1/n,\eps t)>
Ln\right\}\\
&\leq\p\left\{\sup_{t\in[0,T]}M^*(1-1/n,\eps t)> Ln\right\}\\
&\leq\p\left\{\sup_{t\in[0,T]}(w_n(n\eps
t)+1)> Ln\right\}\\
&=\frac{2}{\sqrt{2\pi n\eps T}}\int_{Ln-1}^{\infty}e^{-\frac{x^2}{2n\eps T}}dx
\leq C\exp\left\{-\frac{L^2n}{2\eps T}+\frac{L}{\eps T}\right\},
\end{align*}
where $C$ is independent of $\eps,\ L$ and $n$.

Similarly
$$
\p\{\exists t\in[0,T]:\ y(1/n,\eps t)< -Ln\}\leq C\exp\left\{-\frac{L^2n}{2\eps T}+\frac{L}{\eps
T}\right\}.
$$

Now, for $M>0$ we can estimate
\begin{align*}
  \varlimsup_{\eps\to 0}\eps\ln\p\{\exists t\in[0,T]&:\ y^{\eps}(t)\notin A_L\}\\
&\leq\varlimsup_{\eps\to 0}\eps\ln\left(2C\sum_{n=1}^{\infty}\exp\left\{-\frac{L^2n}{2\eps
T}+\frac{L}{\eps T}\right\}\right)\\
&\leq-\frac{L^2}{2T}+\frac{L}{T}<-M,
\end{align*}
where $L$ is taken large enough. The lemma is proved.
\end{proof}

\begin{lem}
The sequence of processes $\{y^{\eps}\}_{\eps\in(0,1]}$ satisfies $(E2)$, i.e. for every $h\in L_2(\rho)$
the sequence $\{(h,y^{\eps}(\cdot))_{L_2(\rho)}\}_{\eps\in(0,1]}$ is exponentially tight in $C([0,T],\R)$.
\end{lem}

\begin{proof}
To prove the lemma, we will use Corollary~7.1~\cite{MR1483609} (see also Theorem~3~\cite{Schied:1995}). It is enough to show that for each $h$ there
exist positive constants $\alpha$, $\gamma$ and $k$ such that for all $s,t\in[0,T]$, $s<t$
$$
\E\exp\left\{\frac{\gamma}{\eps
(t-s)^{\alpha}}|M_h(\eps t)-M_h(\eps s)|\right\}\leq k^{1/\eps},\quad
\forall \eps\leq\eps_0,
$$
where $M_h(t)=(h,y(t))_{L_2(\rho)},\ t\in[0,T]$.

Using~\eqref{f_mart_characteristic}, for
\begin{align*}
M_h(t)=\int_0^1h(u)y(u,t)\kappa(u)du=\int_0^t(h\kappa,dy(s))_{L_2(\lambda)}
\end{align*}
we have
$$
[ M_h]_t=\int_0^t\|\pr_{y(s)}(h\kappa)\|_{L_2(\lambda)}^2ds\leq
\int_0^t\|h\kappa\|_{L_2(\lambda) }^2ds\leq\|h\|_{L_2(\rho) }^2t.
$$
The inequality for the quadratic variation of $M_h$ and Novikov's theorem imply
\begin{equation}\label{f_expon_mart}
\E\exp\left\{\beta\int_s^t(h\kappa,dy(r))_{L_2(\lambda)}-\frac{\beta^2}{2}\int_s^t\|\pr_{y(r)}
(h\kappa)\|_{L_2(\lambda)}^2dr\right\}=1.
\end{equation}
So, for $\delta>0$
\begin{align*}
\E\exp&\left\{\delta|M_h(\eps t)-M_h(\eps s)|\right\}\leq \E\exp\left\{\delta(M_h(\eps t)-M_h(\eps
s))\right\}\\
&+\E\exp\left\{\delta(M_h(\eps s)-M_h(\eps t))\right\}\\
&=\E\exp\left\{\delta\int_{\eps s}^{\eps t}(h\kappa,dy(r))_{L_2(\lambda)}-\frac{\delta^2}{2}\int_{\eps s}^{\eps t}\|\pr_{y(r)}
(h\kappa)\|_{L_2(\lambda)}^2dr\right.\\
&\left.+\frac{\delta^2}{2}\int_{\eps s}^{\eps t}\|\pr_{y(r)}
(h\kappa)\|_{L_2(\lambda)}^2dr\right\}+\E\exp\left\{\delta(M_{-h}(\eps t)-M_{-h}(\eps
s))\right\}\\
&\leq 2\E\exp\left\{\frac{\eps\delta^2}{2}\|h\|_{L_2(\rho) }^2(t-s)\right\}.
\end{align*}
Taking $\delta=\frac{\sqrt{2}}{\eps(t-s)^{1/2}\|h\|_{L_2(\rho) }}$, we have
$$
\E\exp\left\{\frac{\sqrt{2}|M_h(\eps t)-M_h(\eps s)|}{\eps \|h\|_{L_2(\rho)}
(t-s)^{1/2}}\right\}\leq 2e^{1/\eps}\leq (2e)^{1/\eps}.
$$
This finishes the proof of the lemma.
\end{proof}

From the two previous lemmas we obtain the exponential tightness of $\{y^{\eps}\}_{\eps\in(0,1]}$.
\begin{prp}\label{proposition_exponential_tightness}
 The sequence $\{y^{\eps}\}_{\eps\in(0,1]}$ is exponentially tight in $C([0,T],$ $L_2(\rho))$.
\end{prp}

\subsection{Proof of Theorem~\ref{theorem_LDP}}

We set
$$
L_2^{\uparrow}(\rho)=\{g\in L_2(\rho):\exists \widetilde{g}\in D^{\uparrow},\ g=\widetilde{g},\
\rho\mbox{-a.e.}\}
$$
and 
$$
C_{\id}([0,T],L_2^{\uparrow}(\rho))=\{\varphi\in C([0,T],L_2^{\uparrow}(\rho)):\
\varphi(0)=\id\}.
$$

\begin{rmk}
Since the set $C_{\id}([0,T],L_2^{\uparrow}(\rho))$ is closed in $C([0,T],L_2(\rho))$, it is enough
to state LDP for $\{y^{\eps}\}_{\eps\in(0,1]}$ in the metric space $C_{\id}([0,T],L_2^{\uparrow}(\rho))$.
\end{rmk}

Due to the exponential tightness, for the upper bound it is enough to consider compact sets. According to~\cite{Dembo:2010} (see Theorem~4.1.11), for this it is enough to show that $\Ii$ is a lower-semicontinuous function and
\begin{enumerate}
 \item[(B1)] weak upper bound:
 $$
 \lim_{r\to 0}\varlimsup_{\eps\to 0}\eps\ln\p\{y^{\eps}\in B_r(\varphi)\}\leq-\Ii(\varphi),
 $$
 where $\varphi\in C_{\id}([0,T],L_2^{\uparrow}(\rho))$ and $B_r(\varphi)$ is the open ball in $C_{\id}([0,T],L_2^{\uparrow}(\rho))$ with center
 $\varphi$ and radius  $r$;

 \item[(B2)] lower bound: for every open set $A\subseteq C_{\id}([0,T],L_2^{\uparrow}(\rho))$
 $$
 \varliminf_{\eps\to 0}\eps\ln\p\{y^{\eps}\in A\}\geq-\inf_{\varphi\in A}\Ii(\varphi).
 $$
\end{enumerate}

To prove the upper and lower bounds we will follow the idea in~\cite{Donati_Martin:2008,Donati_Martin:2004} based on exponential change of measure and the Girsanov transformation.

\subsubsection{The upper bound}\label{subsect_The upper bound}
First we check $(B1)$. We set
$$
H=\{h\in C([0,T],L_2(\rho^{-1})):\ \dot{h}\in L_2([0,T],L_2(\rho^{-1}))\},
$$
where $\rho^{-1}(du)=\frac{1}{\kappa(u)}du$.

For $h\in H$ let
\begin{equation}\label{f_martingale_M}
M_t^{\eps,h}=\exp\left\{\frac{1}{\eps}\left[\int_0^t(h(s),dy^{\eps}(s))_{L_2(\lambda)}
-\frac{1}{2}\int_0^t\|\pr_{y^{\eps}(s)}h(s)\|^2_{L_2(\lambda)}ds\right]\right\}.
\end{equation}
By Novikov`s theorem, $M_t^{\eps,h}$, $t\in[0,T]$, is a martingale with $\E M_t^{\eps,h}=1$ (see also~\eqref{f_expon_mart}).
By an integration by parts (Lemma~\ref{lemma_integration_by_parts}), we can write
$$
M_T^{\eps,h}=\exp\left\{\frac{1}{\eps}F(y^{\eps},h)\right\},
$$
where
\begin{align*}
F(\varphi,h)&=(h(T),\varphi(T))_{L_2(\lambda)}-(h(0),\id)_{L_2(\lambda)}\\
&-\int_0^T(\dot{h}(s),\varphi(s))_{L_2(\lambda)}ds\\
&-\frac{1}{2}\int_0^T\|\pr_{\varphi(s)}h(s)\|^2_{
L_2(\lambda)}ds, \quad\varphi\in C_{\id}([0,T],L_2^{\uparrow}(\rho)).
\end{align*}

For $\varphi\in C_{\id}([0,T],L_2^{\uparrow}(\rho))$ we have
\begin{align*}
 \p\{y^{\eps}\in B_r(\varphi)\}&=\E\left[\I_{\{y^{\eps}\in
B_r(\varphi)\}}\frac{M_T^{\eps,h}}{M_T^{\eps,h}}\right]\\
&\leq\exp\left\{-\frac{1}{\eps}\inf_{\psi\in B_r(\varphi)}F(\psi,h)\right\}\E M_T^{\eps,h}\\
&=\exp\left\{-\frac{1}{\eps}\inf_{\psi\in B_r(\varphi)}F(\psi,h)\right\}.
\end{align*}

Using the inequality $\|\pr_{\varphi(s)}h(s)\|^2_{L_2(\lambda)}\leq\|h(s)\|^2_{L_2(\lambda)}$, we obtain
\begin{align*}
 \varlimsup_{\eps\to 0}\eps\ln\p\{y^{\eps}\in B_r(\varphi)\}\leq -\inf_{\psi\in
B_r(\varphi)}F(\psi,h)\leq -\inf_{\psi\in B_r(\varphi)}\Phi(\psi,h),
\end{align*}
where
\begin{align*}
\Phi(\varphi,h)&=(h(T),\varphi(T))_{L_2(\lambda)}-(h(0),\id)_{L_2(\lambda)}\\
&-\int_0^T(\dot{h}(s),\varphi(s))_{L_2(\lambda)}ds\\
&-\frac{1}{2}\int_0^T\|h(s)\|^2_{L_2(\lambda)
}ds, \quad\varphi\in C_{\id}([0,T],L_2^{\uparrow}(\rho)).
\end{align*}
Since the map $\Phi(\varphi,h)$, $\varphi\in C_{\id}([0,T],L_2^{\uparrow}(\rho))$, is continuous for fixed $h$,
$$
\lim_{r\to 0}\varlimsup_{\eps\to 0}\eps\ln\p\{y^{\eps}\in B_r(\varphi)\}\leq -\Phi(\varphi,h).
$$
Minimizing in $h\in H$, we obtain
$$
\lim_{r\to 0}\varlimsup_{\eps\to 0}\eps\ln\p\{y^{\eps}\in B_r(\varphi)\}\leq
-\sup_{h\in H}\Phi(\varphi,h).
$$

Now $(B1)$ will follow from the following.

\begin{prp}\label{proposition_sup_of_F}
 For each $\varphi\in C_{\id}([0,T],L_2^{\uparrow}(\rho))$
$$
\sup_{h\in H}\Phi(\varphi,h)=\Ii(\varphi).
$$
\end{prp}

\begin{proof}
First we prove the assertion of the proposition for $\varphi$ satisfying
$$
J(\varphi):=\sup_{h\in H}\Phi(\varphi,h)<\infty.
$$
Replacing $h$ by $\theta h$, $\theta\in\R$, and maximizing the expression $\Phi(\varphi,\theta h)$  over $\theta$ for fixed $h$, we get
\begin{equation}\label{f_the_norm_of_G}
J(\varphi)=\frac{1}{2}\sup_{h\in H}\frac{G^2(\varphi,h)}{\int_0^T\|h(s)\|^2_{L_2(\lambda)
}ds}<\infty.
\end{equation}
By~\eqref{f_the_norm_of_G} and Lemma~\ref{lemma_density_of_H} the linear map
$$
G_{\varphi}:h\to G(\varphi,h)
$$
can be extended to the space $L_2([0,T],$ $L_2(\lambda))$ and consequently,
\begin{equation}\label{f_represent_of_G}
G(\varphi,h)=\int_0^T(k_{\varphi}(s),h(s))ds
\end{equation}
for some function $k_{\varphi}\in L_2([0,T],$ $L_2(\lambda))$. Thus, by the integration by parts formula for Bochner integrals, $\varphi$ is absolutely continuous and
$\dot{\varphi}=k_{\varphi}$. Applying the Cauchy-Schwarz inequality to~\eqref{f_represent_of_G} we
get
\begin{align*}
G(\varphi,h)^2&\leq 2\Ii(\varphi)\int_0^T\|h(s)\|^2_{L_2(\lambda)}ds.
\end{align*}
This yields $J(\varphi)\leq \Ii(\varphi)$
and since $H$ is dense in $L_2([0,T],L_2(\lambda))$, we get the equality $J(\varphi)=\Ii(\varphi)$.

If $\Ii(\varphi)<\infty$, then $\varphi$ is absolutely continuous and $k_{\varphi}=\dot{\varphi}$
in~\eqref{f_represent_of_G}. So, $J(\varphi)\leq \Ii(\varphi)<\infty$ and consequently we have
$J(\varphi)=\Ii(\varphi)$. This completes the proof of the proposition.
\end{proof}

\begin{cor}
 $\Ii$ is lower-semicontinuous as supremum of continuous functions.
\end{cor}

\subsubsection{The lower bound}\label{lower_bound}
In order to obtain the lower bound $(B2)$, it is enough to find a subset $\RR\subset
C_{\id}([0,T],L_2^{\uparrow}(\rho))$
such that for each $\varphi\in\RR$
\begin{equation}\label{f_lower_bound}
\lim_{r\to 0}\varliminf_{\eps\to 0}\eps\ln\p\{y^{\eps}\in B_r(\varphi)\}\geq -\Ii(\varphi),
\end{equation}
and prove that for each $\varphi$ satisfying $\Ii(\varphi)<\infty$, there exists a sequence
$\{\varphi_n\}\subset\RR$ such that $\varphi_n\to\varphi$ in $C_{\id}([0,T],L_2^{\uparrow}(\rho))$
and
$\Ii(\varphi_n)\to \Ii(\varphi)$.

We denote 
\begin{equation}\label{f_D_two_arrows}
D^{\uparrow\uparrow}=\{g\in D([0,1],\R):\ \forall u<v\in[0,1],\ g(u)<g(v)\}
\end{equation}
and define $L_2^{\uparrow\uparrow}(\rho)$ in the same way as $L_2^{\uparrow}(\rho)$, replacing
$D^{\uparrow}$ by $D^{\uparrow\uparrow}$. Set
\begin{align*}
\RR=&\bigg\{\varphi\in C([0,T], L_2^{\uparrow\uparrow}(\lambda)):\ \Ii(\varphi)<\infty,\ \dot{\varphi}\in
H_{L_2(\lambda)},\\
&\dot{\varphi}\ \mbox{is continuous in } (u,t)\ \mbox{and } \varphi(u,t)\ \mbox{is continuously differentiable in}\ u\\
&\mbox{with bounded (uniformly in } t,u \mbox{) derivative}\ \frac{\partial\varphi(u,t)}{\partial u}\bigg\},
\end{align*}
where
$$
H_{L_2(\lambda)}=\{h\in C([0,T],L_2(\lambda)):\ \dot{h}\in L_2([0,T],L_2(\lambda))\}.
$$
For $h\in H_{L_2(\lambda)}$ define the new probability measure $\p^{\eps,h}$ with density
$$
\frac{d\p^{\eps,h}}{d\p}=M_T^{\eps,h},
$$
where $M_T^{\eps,h}$ is defined by~\eqref{f_martingale_M}. By Novikov's theorem and Theorem~\ref{theorem_Girsanov}, the random element $y^{\eps}$ in
$D([0,1],C[0,T])$
satisfies (w.r.t. $\p^{\eps,h}$) the following properties
\begin{enumerate}
\item[$(D^{\eps}1)$] for all $u\in[0,1]$ the process
$$
\eta^{\eps}(u,\cdot)=y^{\eps}(u,\cdot)-\int_0^{\cdot}\left(\pr_{y^{\eps}(s)}h(s)\right)(u)ds
$$
is a continuous local square integrable $(\F_{\eps t})$-martingale;

\item[$(D^{\eps}2)$] for all $u\in[0,1]$, $y^{\eps}(u,0)=u$;

\item[$(D^{\eps}3)$] for all $u<v$ from $[0,1]$ and $t\in[0,T]$, $y^{\eps}(u,t)\leq y^{\eps}(v,t)$;

\item[$(D^{\eps}4)$]  for all $u,v\in[0,1]$ and $t\in[0,T]$,
$$
[
\eta^{\eps}(u,\cdot),\eta^{\eps}(v,\cdot)]_t=\eps\int_0^t\frac{\I_{\{\tau^{\eps}_{u,v}\leq
s\}}ds}{m^{\eps}(u,s)},
$$
where $\tau^{\eps}$ and $m^{\eps}$ is defined in the same way as $\tau$ and $m$, replacing $y$ by
$y^{\eps}$.
\end{enumerate}

Note that if $\varphi\in\RR$, then
\begin{equation}\label{f_lim_of_prob}
\lim_{\eps\to 0}\p^{\eps,\dot{\varphi}}\{y^{\eps}\in
B_r(\varphi)\}=1
\end{equation}
for all $r>0$, by Proposition~\ref{proposition_conv_in_probability}.

Let us fix $\varphi\in\RR$ and set $h=\dot{\varphi}$. Noting that $y^{\eps}\in L_2([0,T],L_2(\lambda))$ a.s., we estimate
\begin{align*}
&\p\{y^{\eps}\in B_r(\varphi)\}=\E^{\eps,h}\frac{\I_{\{y^{\eps}\in B_r(\varphi)\}}}{M_T^{\eps,h}}\\
&\geq\exp\left\{-\frac{1}{\eps}\sup_{\psi\in B_r(\varphi)\cap L_2([0,T],L_2(\lambda))}F(\psi,h)\right\}
\p^{\eps,h}\{y^{\eps}\in B_r(\varphi)\},
\end{align*}
where $\E^{\eps,h}$ denotes the expectation w.r.t. $\p^{\eps,h}$. Thus, by~\eqref{f_lim_of_prob},
\begin{equation}\label{f_inequality_for_LDP}
\varliminf_{\eps\to 0}\eps\ln\p\{y^{\eps}\in B_r(\varphi)\}\geq -\sup_{\psi\in B_r(\varphi)\cap
L_2([0,T],L_2(\lambda))}F(\psi,h).
\end{equation}

Next we prove the continuity of the map $g\mapsto\pr_gf$ on $L_2^{\uparrow\uparrow}(\rho)$ for each $f\in L_2(\lambda)$.

\begin{lem}\label{lemma_semi_cont_of_F}
Let $g\in L_2^{\uparrow\uparrow}(\rho)$ and $f\in L_2(\lambda)$. If a
sequence $\{g_n\}_{n\geq 1}$ of elements $L_2^{\uparrow}(\rho)$ converges to $g$ a.e., then
$\{\pr_{g_n}f\}_{n\geq 1}$ converges to $f$ in $L_2(\lambda)$ and 
\begin{equation}\label{f_semi_cont_of_F}
\lim_{n\to\infty}\|\pr_{g_n}f\|_{L_2(\lambda)}= \|f\|_{L_2(\lambda)}.
\end{equation}
\end{lem}

\begin{proof}
First we note that $\sigma(g)$ is a Borel $\sigma$-algebra on $[0,1]$. Moreover, since  $\{g_n\}_{n\geq 1}$ converges to $g$ a.e., one can show that for almost all $a,b\in[0,1]$ and $a<b$ there exists a sequence $\{c_n,\ d_n,\ n\in\N\}$ such that $\lambda((a,b)\bigtriangleup g_n^{-1}(c_n,d_n))\to 0$ as $n\to\infty.$ This immediately implies that for all $A\in\sigma(g)$ there exist $A_n\in\sigma(g_n)$, $n\in\N$, such that $\lambda(A\bigtriangleup A_n)\to 0$. Thus, by Proposition~1~\cite{Alonso:1998}, $\pr_{g_n}f\to\pr_gf=f$ in $L_2$. Consequently, we also have $\|\pr_{g_n}f\|_{L_2(\lambda)}\to\|f\|_{L_2(\lambda)}$.  The lemma is proved.
\end{proof}

So, by Lemma~\ref{lemma_semi_cont_of_F}, \eqref{f_inequality_for_LDP} yields
$$
\lim_{r\to 0}\varliminf_{\eps\to 0}\eps\ln\p\{y^{\eps}\in B_r(\varphi)\}\geq
-F(\varphi,h)=-\Ii(\varphi).
$$
Here the last equality follows from the form of the map $F$, the choice of $h$ and Lemma~\ref{lemma_integration_by_parts}.

\begin{prp}\label{proposition_aproxim_of_I}
For each $\varphi\in\h$, where $\h$ is defined by~\eqref{f_set_H}, satisfying $\Ii(\varphi)<\infty$, there exists a sequence
$\{\varphi_n\}\subset\RR$ such that $\varphi_n\to\varphi$ in $C_{\id}([0,T], L_2^{\uparrow}(\rho))$
and
$\Ii(\varphi_n)\to \Ii(\varphi)$.
\end{prp}

\begin{proof} 
Firs we note that it is enough to check the statement only for functions
$\varphi$ with bounded derivative,
i.e.
$$
\sup_{(u,t)\in[0,1]\times[0,T]}|\dot{\varphi}(u,t)|<\infty.
$$
Then the proposition can be proved using the approximation of $\dot{\varphi}$ in $L_2([0,T],L_2(\lambda))$ by functions 
$$
\psi_{\delta,\alpha}(u,t)=u+\int_0^t\widetilde{\varphi}\ast\varsigma_{\delta}(u,s)ds+\alpha tu,\quad
u\in[0,1],\ t\in[0,T],\ \alpha,\delta>0, 
$$
where 
\begin{align*}
\widetilde{\varphi}(u,t)=\begin{cases}
                     \dot{\varphi}(u,t),&(u,t)\in[0,1]\times[0,T],\\
                     C,                &(u,t)\in(1,2]\times[0,T],\\
                     -C,                &(u,t)\in[-1,0)\times[0,T],\\
                     0,                 &\mbox{otherwise},
                   \end{cases}
\end{align*} 
$
C=\sup_{(u,t)\in[0,1]\times[0,T]}|\dot{\varphi}(u,t)|
$
and $\varsigma_{\delta}(u,t)$, $(u,t)\in\R^2$, is a standard mollifier.
\end{proof}

\begin{proof}[Proof of Theorem \ref{thm_varadhan_formula}] 
First of all we note that the family $\{y(\eps)\}_{\eps\in(0,T]}$ satisfies large deviations in $L_2(\rho)$ with the rate function
$\frac{1}{2}\|\id-\cdot\|_{L_2(\lambda)}^2$ (for simplicity of notation we suppose that $\|\id-g\|_{L_2(\lambda)}=+\infty$ whenever $g\notin L_2(\lambda)$). This immediately follows from Theorem~\ref{theorem_LDP} and the contraction principle for large deviations. Next, let us show that for each convex closed set $C$ in $L_2(\rho)$ with non-empty interior we have
\begin{equation}\label{f_rate_function_for_y}
\lim_{\eps\to 0}\eps\ln\p\{y(\eps)\in C\}=-\frac{1}{2}\inf_{g\in C}\|\id-g\|_{L_2(\lambda)}^2.
\end{equation}
To prove this, it is enough to show that
\begin{equation}\label{f_inequality_inf}
\inf_{g\in C^{\circ}}\|\id-g\|_{L_2(\lambda)}\leq\inf_{g\in C}\|\id-g\|_{L_2(\lambda)},
\end{equation}
where $C^{\circ}$ denotes the interior of $C$ in $L_2(\rho)$. Let $\inf_{g\in C}\|\id-g\|_{L_2(\lambda)}<\infty$ and $\delta\in(0,1]$ be fixed. Then there exists $g_0\in L_2(\lambda)\cap C$ such that 
$$
\|\id-g_0\|_{L_2(\lambda)}\leq\inf_{g\in C}\|\id-g\|_{L_2(\lambda)}+\delta.
$$

Since $C^{\circ}$ is non-empty, there exists $g_1\in C^{\circ}$ and $r>0$ such that $B(g_1,r):=\{g\in L_2(\rho):\ \|g-g_1\|_{L_2(\rho)}<r\}\subseteq C$. Moreover, $g_1$ can be chosen from $L_2(\lambda)$ because $L_2(\lambda)$ is dense in $L_2(\rho)$. Next, let $g_c=(1-\delta_0)g_0+\delta_0g_1$, where $\delta_0=\frac{\delta}{1+\|g_0+g_1\|_{L_2(\lambda)}}$. Using the convexity of $C$, we can see that $B(g_c,\delta_0r)\subseteq C$ and consequently, $g_c$ belongs to $C^{\circ}$. Now we can estimate
\begin{align*}
 \inf_{g\in C^{\circ}}\|\id-g\|_{L_2(\lambda)}&\leq \|\id-g_c\|_{L_2(\lambda)}\leq\|\id-g_0\|_{L_2(\lambda)}+\|g_0-g_c\|_{L_2(\lambda)}\\
 &\leq\|\id-g_0\|_{L_2(\lambda)}+\delta_0\|g_0+g_1\|_{L_2(\lambda)}\\
 &<\inf_{g\in C}\|\id-g\|_{L_2(\lambda)}+2\delta.
\end{align*}
Making $\delta\to 0$, we obtain~\eqref{f_inequality_inf}.

Thus, the Varadhan formula~\eqref{f_Varadhan_formula} is obtained from a straightforward 
combination~\eqref{f_rate_function_for_y},  the contraction principle for large 
deviations applied to the endpoint map
$C([0,1],
L_2(\rho)) \to L_2(\rho)$, $\mu_{t\in [0,1]} \to \mu_1$ and  the fact
that the map
\[ \iota: \,  D^{\uparrow}([0,1]) \ni  g \mapsto   g_\#\leb \in \mathcal
P(\mathbb R)  \]
is an isometry from the $L_2(\lambda)$-metric to the quadratic Wasserstein 
metric~$d_\mathcal W$ (see e.g. Section~2.1~\cite{Brenier:2013}).
\end{proof}

\appendix

\section{Some properties of absolutely continuous functions}

The following lemma follows from the definition of the stochastic integral, given in Section~\ref{stochastic_integral}, and the integration by parts formula for integrals with respect to real values continuous martingales.

\begin{lem}\label{lemma_integration_by_parts}
For every absolutely continuous function $f(t),\ t\in[0,T]$, with values in $L_2(\lambda)$
\begin{equation}\label{f_integr_by_parts}
\begin{split}
\int_0^t(f(s),dy(s))&=(f(t),y(t))_{L_2(\lambda)}-(f(0),y(0))_{L_2(\lambda)}\\
&-\int_0^t(\dot{f}(s),y(s))_{L_2(\lambda)}ds, \quad t\in[0,T],
\end{split}
\end{equation}
almost surely, where the integral in the left hand side was defined in Section~\ref{stochastic_integral}.
\end{lem}

\begin{lem}\label{lemma_density_of_H}
 The set $H$, which is defined in Subsection~\ref{subsect_The upper bound}, i.e.
$$
H=\{h\in C([0,T],L_2(\rho^{-1})):\ \dot{h}\in L_2([0,T],L_2(\rho^{-1}))\},
$$
where $\rho^{-1}(du)=\frac{1}{\kappa(u)}du$, is dense in $L_2([0,T],L_2(\lambda))$.
\end{lem}

\begin{proof}
The lemma can be proved using the density of the space of continuously differentiable functions $C^1([0,T])$ on $[0,T]$ in $L_2([0,T],\R)$ and the approximation of each function $f\in L_2([0,T],L_2(\rho^{-1}))$ by $f_n=\sum_{k=1}^n\widetilde{f}_ke_k$, $n\in\N$, where $\{e_n\}_{n\geq 1}$ is an orthonormal basis in $L_2(\rho^{-1})$ and $\widetilde{f}_k(t)=(f(t),e_k)_{L_2(\rho^{-1})}$, $t\in[0,T]$.
\end{proof}

\section{Convergence of the flow of particles with drift}

In this section we prove that the process $\{z^{\eps}\}_{\eps\in(0,1]}$ satisfying $(D^{\eps}1)-(D^{\eps}4)$ with $h=\dot{\varphi}$ tends to
$\varphi$. Note that $z^{\eps}$ is a weak martingale solution to the equation
$$
dz^{\eps}(t)=\pr_{z^{\eps}}\dot{\varphi}(t)dt+\sqrt{\eps}\pr_{z^{\eps}}dW_t.
$$
If we show that $z^{\eps}$ converges to a process $z$ taking values from $L_2^{\uparrow\uparrow}(\rho)$, then by Lemma~\ref{lemma_semi_cont_of_F}, $z$ should be a solution of the equation
$$
dz(t)=\dot{\varphi}(t)dt,
$$
It gives $z=\varphi$.

Thus, we prove first that the family $\{z^{\eps}\}_{\eps\in(0,1]}$ is tight. Then we show that any limit point $z$ of $\{z^{\eps}\}_{\eps\in(0,1]}$ is $L_2^{\uparrow\uparrow}(\rho)$-valued process. As we noted, it immediately gives $z=\varphi$. Since $\{z^{\eps}\}_{\eps\in(0,1]}$ has only one nonrandom limit point, we obtain that $\{z^{\eps}\}_{\eps\in(0,1]}$ tends to $\varphi$ in probability (not only in distribution).

\begin{prp}\label{proposition_conv_in_probability}
 Let $\varphi\in\RR$ and a family of random elements $\{z^{\eps}\}_{\eps\in(0,1]}$
satisfies properties $(D^{\eps}1)-(D^{\eps}4)$ with $h=\dot{\varphi}$, then $z^\eps$ tends to
$\varphi$ in the
space $C([0,T], L_2(\rho))$ in probability.
\end{prp}

To prove the proposition, we first establish tightness of $\{z^{\eps}\}$ in $C([0,T],$ $L_2(\rho))$,
using the boundedness of $\dot{\varphi}$ and the same argument as in the proof of exponential
tightness of $\{y^{\eps}\}$. Next, testing the convergent subsequence $\{z^{\eps'}\}$ by functions
$l$ from $C([0,1]\times[0,T],\R)$ and using integration by parts we will obtain
\begin{align*}
 \int_0^T\int_0^1l(u,t)(z^{\eps'}(u,t)-u)dtdu&=
\int_0^T\int_0^1L(u,t)(\pr_{z^{\eps'}(t)}\dot{\varphi}(t))(u)dtdu\\
&+\int_0^T\int_0^1L(u,t)d\eta^{\eps'}(u,t)du,
\end{align*}
where $L(u,t)=\int_t^Tl(u,s)ds$ and $z^{\eps'}\to z$. If $z(t)$ belongs to
$L_2^{\uparrow\uparrow}(\rho)$ for all $t\in[0,T]$, then passing to the limit and using
Lemma~\ref{lemma_semi_cont_of_F} we obtain
$$
\int_0^T\int_0^1l(u,t)(z(u,t)-u)dtdu=\int_0^T\int_0^1L(u,t)\dot{\varphi}(u,t)dtdu,
$$
which implies $z=\dot{\varphi}$.

The fact that $z(t)\in L_2^{\uparrow\uparrow}(\rho)$ will follow from the following lemma.

\begin{lem}\label{lemma_coalescing_of_particles}
Let $\varphi$ and $\{z^{\eps}\}$ be such as in Proposition~\ref{proposition_conv_in_probability}.
Then for each $u<v$ there exists $\delta>0$ such that
$$
\lim_{\eps\to 0}\p\{z^{\eps}(v,t)-z^{\eps}(u,t)\leq\delta\}=0.
$$
\end{lem}

Let $\Ss(u,v,t)$ be a finite set of intervals contained in $[u,v]$ for all $t\in(0,T]$ such that
\begin{enumerate}
 \item[1)] if $\pi_1,\pi_2\in\Ss(u,v,t)$ and $\pi_1\neq \pi_2$, then $\pi_1\cap\pi_2=\emptyset$;

 \item[2)] $\bigcup\Ss(u,v,t)=[u,v]$;

 \item[3)] for all $s<t$ and $\pi_1\in\Ss(u,v,s)$ there exists $\pi_2\in\Ss(u,v,t)$ that
contains $\pi_1$;

 \item[4)] there exists decreasing sequence $\{t_n\}_{n\geq 1}$ on $(0,T]$ that tends to $0$ and
$$
\Ss(u,v,t)=\Ss(u,v,t_n),\quad t\in[t_n,t_{n-1}),\ \ n\in\N,\ \ t_0=T;
$$

 \item[5)] for each monotone sequence $\pi(t)\in\Ss(u,v,t)$, $t>0$, $\bigcap_{t>0}\pi(t)$ is a one-point
set.
\end{enumerate}

\begin{lem}\label{lemma_properti_of_drift}
Let $\varphi\in\RR$ and $[\widetilde{u},\widetilde{v}]\subset(0,1)$. Then there exists $\gamma>0$
such that for each interval $(u,v)\supset[\widetilde{u},\widetilde{v}]$ there exist
$u_0\in(u,\widetilde{u})$ and $v_0\in(\widetilde{v},v)$:
\begin{align*}
\inf_{t\in[0,T]}\left[v_0-u_0+\int_0^t(\pr_{\Ss(s)}\dot{\varphi}(s))(v_0)ds-
\int_0^t(\pr_{\Ss(s)}\dot{\varphi}(s))(u_0)ds\right]=\delta>0,
\end{align*}
 for all $\Ss(t)=\Ss(0\vee(u-\gamma),(v+\gamma)\wedge 1,t)$, $t\in(0,T]$, such that $u_0$ and $v_0$
belong to separate intervals from $\Ss(T)$, and $\pr_{\Ss(t)}$ denotes the projection in $L_2(\lambda)$
onto the space of $\sigma(\Ss(t))$-measurable functions.
\end{lem}

\begin{proof}
 Let $u\in[0,1]$ and $\Ss(t)=\Ss(0,1,t)$, $t\in[0,T]$. Then we can choose a sequence of intervals
$\{\pi_n\}_{n\geq 1}$ and a decreasing sequence $\{s_n\}_{n\geq 1}$ from
$(0,T]$ converging to $0$ such that $\pi_{n+1}\subseteq\pi_n\subseteq [0,1]$,
$\{u\}=\bigcap_{n=1}^{\infty}\pi_n$ and
\begin{align*}
u+\int_0^t(\pr_{\Ss(s)}\dot{\varphi}(s))(u)ds&=u+\sum_{n=1}^{\infty}\int_{s_n\wedge
t}^{s_{n-1}\wedge t} \left(\frac{1}{|\pi_n|}\int_{\pi_n}\dot{\varphi}(q,r)dq\right)dr\\
&=u+\sum_{n=1}^{\infty}\frac{1}{|\pi_n|}\int_{\pi_n}(\varphi(q,s_{n-1}\wedge t)-\varphi(q,s_n\wedge
t)dq\\
&=\frac{1}{|\pi_k|}\int_{\pi_k}\varphi(q,t)dq+\sum_{n=k}^{\infty}\left[\frac{1}{|\pi_{n+1}|}
\int_{\pi_{n+1}}\varphi(q,s_n)dq\right.\\
&\left.-\frac{1}{|\pi_n|}\int_{\pi_n}\varphi(q,s_n)dq\right],
\end{align*}
where $t\in[s_k,s_{k-1})$.

We estimate the $n$-th term of the sum. For convenience of calculations, let $[a,b]\subseteq [c,d]$ and
$f:[0,1]\to\R$
be non-decreasing absolutely continuous function with bounded derivative. So,
\begin{align*}
 \frac{1}{b-a}\int_a^bf(x)dx&-\frac{1}{d-c}\int_c^df(x)dx\leq
\frac{1}{b-a}\int_a^bf(x)dx-\frac{1}{b-c}\int_c^bf(x)dx\\
&=\frac{a-c}{(b-a)(b-c)}\int_a^bf(x)dx-\frac{1}{b-c}\int_c^af(x)dx\\
&\leq\frac{a-c}{b-c}f(b)-\frac{a-c}{b-c}f(c)=\frac{a-c}{b-c}(f(b)-f(c))\\
&\leq\sup_{x\in[0,1]}\dot{f}(x)(a-c).
\end{align*}

Taking $c=a_n$, $d=b_n$, $a=a_{n+1}$, $b=b_{n+1}$ and $f=\varphi(\cdot,s_n)$, where
$a_n<b_n$ are the ends of $\pi_n$, we get
\begin{align*}
u+\int_0^t(\pr_{\Ss(s)}\dot{\varphi}(s))(u)ds&\leq
\frac{1}{|\pi_k|}\int_{\pi_k}\varphi(q,t)dq+\sum_{n=k}^{\infty}\widetilde{C}(a_{n+1}-a_n)\\
&\leq\frac{1}{|\pi_k|}\int_{\pi_k}\varphi(q,t)dq+\widetilde{C}(u-a_k),
\end{align*}
where $\widetilde{C}=\sup_{(u,t)\in[0,1]\times[0,T]}\frac{\partial\varphi}{\partial u}(u,t)$.

Similarly, we can obtain
$$
u+\int_0^t(\pr_{\Ss(s)}\dot{\varphi}(s))(u)ds\geq\frac{1}{|\pi_k|}\int_{\pi_k}\varphi(q,
t)dq-\widetilde{C}(b_k-u).
$$

Next, let $\widetilde{u}$ and $\widetilde{v}$ is from the statement of the lemma. Since $\varphi$
is continuous on $[0,1]\times[0,T]$ and increasing by the first argument, the function
$$
G(a,b,t)=\frac{1}{\widetilde{v}-b}\int_b^{\widetilde{v}}\varphi(q,t)dq
-\frac{1}{a-\widetilde{u}}\int_{\widetilde{u}}^a\varphi(q,t)dq
$$
is positive and continuous on $E=\{(a,b,t):\widetilde{u}\leq a\leq b\leq \widetilde{v},\ t\in[0,T]\}$. Hence
$$
\delta_1=\inf_{E}G>0.
$$
Take $\gamma=\frac{\delta_1}{8\widetilde{C}}$ and for $u,v\in[0,1]$ satisfying $(u,v)\supset
[\widetilde{u},\widetilde{v}]$ set
\begin{align*}
 u_0&=(u+\gamma)\wedge\widetilde{u},\\
 v_0&=(v-\gamma)\vee\widetilde{v}.
\end{align*}

Let $\Ss(t)=\Ss(0\vee(u-\gamma),(v+\gamma)\wedge 1,t)$, $t\in(0,T]$, such that $u_0$ and $v_0$
belong to separate intervals from $\Ss(T)$ and $t\in(0,T]$ be fixed. For $\pi_1,\pi_2$ belonging to
$\Ss(t)$ and containing $u_0,\ v_0$ respectively, we obtain
\begin{align*}
v_0-u_0&+\int_0^t(\pr_{\Ss(s)}\dot{\varphi}(s))(v_0)ds-\int_0^t(\pr_{\Ss(s)}\dot{\varphi}
(s))(u_0)ds\\
&\geq\frac{1}{|\pi_2|}\int_{\pi_2}\varphi(q,t)dq-\frac{1}{|\pi_1|}\int_{\pi_1}\varphi(q,t)dq\\
&-\widetilde{C}(d-v_0)-\widetilde{C}(u_0-a)\\
&\geq G(b\vee\widetilde{u},c\wedge\widetilde{v},t)-\widetilde{C}(d-v_0)-
\widetilde{C}(u_0-a),
\end{align*}
where $c<d$ and $a<b$ are the ends of $\pi_1$ and $\pi_2$ respectively and $b\leq c$ because $u_0$
and $v_0$ belong to separate intervals from $\Ss(T)$. Since $u-\gamma\leq a<u_0\leq u+\gamma$
and $v-\gamma\leq v_0<b\leq v+\gamma$, we have
\begin{align*}
v_0-u_0&+\int_0^t(\pr_{\Ss(s)}\dot{\varphi}(s))(v_0)ds-\int_0^t(\pr_{\Ss(s)}\dot{\varphi}
(s))(u_0)ds\\
&\geq \delta_1-\widetilde{C}(v+\gamma-v+\gamma)-\widetilde{C}(u+\gamma-u+\gamma)\\
&=\delta_1-4\widetilde{C}\gamma=\frac{\delta_1}{2}>0.
\end{align*}
It finishes the proof of the lemma.
\end{proof}

\begin{proof}[Proof of Lemma~\ref{lemma_coalescing_of_particles}]
Let $u<v$ be a fixed points from $(0,1)$ and $\delta$, $u_0$, $v_0$, $\gamma$ be defined in
Lemma~\ref{lemma_properti_of_drift} for some $[\widetilde{u},\widetilde{v}]\subset (u,v)$. Suppose
that
$$
\varliminf_{\eps\to 0}\p\left\{z^{\eps}(v,t)-z^{\eps}(u,t)\leq\frac{\delta}{2}\right\}>0.
$$

Set
\begin{align*}
 B_1^{\eps}&=\{z^{\eps}(u,t)=z^{\eps}((u-\gamma)\vee 0,t)\},\\
 B_2^{\eps}&=\{z^{\eps}(v,t)=z^{\eps}((v+\gamma)\wedge 1,t)\},\\
 A^{\eps}&=\left\{z^{\eps}(v,t)-z^{\eps}(u,t)\leq\frac{\delta}{2}\right\}.
\end{align*}

Since the diffusion rate of $z^{\eps}(u,\cdot)$ grows to infinity as the time tends to 0, it is
convenient to work with the mean of $z^{\eps}$ because we can control the growing of
diffusion rate in this case. So, denote
\begin{align*}
 \xi_u^{\eps}(t)&=\frac{1}{u_0-u}\int_u^{u_0}z^{\eps}(q,t)dq,\\
 \xi_v^{\eps}(t)&=\frac{1}{v-v_0}\int_{v_0}^uz^{\eps}(q,t)dq.
\end{align*}
It is easy to see that
$$
\widetilde{A}^{\eps}=\left\{\xi_v^{\eps}(t)-\xi_u^{\eps}(t)\leq\frac{\delta}{2}\right\}\supseteq
A^{\eps}.
$$

Next, using the processes $\xi_u^{\eps}$ and $\xi_v^{\eps}$, we want to show that
\begin{equation}\label{f_lim_of_A_with_B}
\lim_{\eps\to 0}\p\{A^{\eps}\cap(B_1^{\eps}\cup B_2^{\eps})^c\}=0.
\end{equation}

Note that $\xi_u^{\eps}$ and
$\xi_v^{\eps}$ are diffusion processes, namely
\begin{align*}
  \xi_u^{\eps}(t)&=\frac{u_0+u}{2}+\int_0^ta_u^{\eps}(s)ds+\chi_u^{\eps}(t),\\
  \xi_v^{\eps}(t)&=\frac{v_0+v}{2}+\int_0^ta_v^{\eps}(s)ds+\chi_v^{\eps}(t),
\end{align*}
where
\begin{align*}
  a_u^{\eps}(t)&=\frac{1}{u_0-u}\int_u^{u_0}(\pr_{z^{\eps}(t)}\dot{\varphi}(t))(q)dq,\\
  a_v^{\eps}(t)&=\frac{1}{v-v_0}\int_{v_0}^v(\pr_{z^{\eps}(t)}\dot{\varphi}(t))(q)dq,\\
  \chi_u^{\eps}(t)&=\frac{1}{u_0-u}\int_0^t\int_u^{u_0}d\eta^{\eps}(q,s)dq,\\
  \chi_v^{\eps}(t)&=\frac{1}{v-v_0}\int_0^t\int_{v_0}^vd\eta^{\eps}(q,s)dq.
\end{align*}
By choosing of $u_0,\ v_0$ and $\delta$, we have
\begin{align*}
L^{\eps}(t,\omega)&=\frac{v_0+v}{2}-\frac{u_0+u}{2}\\
&+\int_0^ta_v^{\eps}(s,\omega)ds-\int_0^ta_u^{\eps}(s,\omega)ds\geq\delta,\quad t\in[0,T],\ \
\omega\in(B_1^{\eps}\cup B_2^{\eps})^c.
\end{align*}

Denote the difference $\xi_v^{\eps}-\xi_u^{\eps}$ by $\xi^{\eps}$. Note that the quadratic variation of the martingale part $\chi^{\eps}$ of $\xi^{\eps}$ satisfies
$$
[\chi^{\eps}]_t\leq\eps Ct,
$$
where $C=\frac{1}{u_0-u}+\frac{1}{v-v_0}$. So, denoting
$$
\sigma^{\eps}=\inf\left\{t:\ \xi^{\eps}(t)=\frac{\delta}{2}\right\},
$$
we get
$$
\p\{A^{\eps}\cap(B_1^{\eps}\cup B_2^{\eps})^c\}\leq\p\{\{\sigma^{\eps}\leq t\}\cap(B_1^{\eps}\cup
B_2^{\eps})^c\},
$$
which implies~\eqref{f_lim_of_A_with_B}. Indeed, by Theorem~2.7.2~\cite{Watanabe:1981:en}, there
exists a standard Wiener process $w^{\eps}(t),\ t\geq 0$, such that
$$
\chi^{\eps}(t)=w^{\eps}([\chi^{\eps}]_t).
$$
Define $\tau^{\eps}=\inf\left\{t:w^{\eps}(\eps Ct)=-\frac{\delta}{2}\right\}$. Since
$$
\xi^{\eps}(t)=\delta+(L^{\eps}(t)-\delta)+\chi^{\eps}(t)
$$
and the term $L^{\eps}-\delta$ is non-negative on $(B_1^{\eps}\cup B_2^{\eps})^c$, the process
$\delta+w^{\eps}(\eps C\cdot)$ hits at the point $\frac{\delta}{2}$ sooner than $\xi^{\eps}$.
Consequently, $\{\sigma^{\eps}\leq t\}\cap(B_1^{\eps}\cup B_2^{\eps})^c\subseteq \{\tau^{\eps}\leq
t\}\cap(B_1^{\eps}\cup B_2^{\eps})^c$. This yields~\eqref{f_lim_of_A_with_B}.

Next, the relation $\p\{A^{\eps}\cap(B_1^{\eps}\cup B_2^{\eps})\}
=\p\{A^{\eps}\}-\p\{A^{\eps}\cap(B_1^{\eps}\cup B_2^{\eps})^c\}$,~\eqref{f_lim_of_A_with_B} and the assumption $\varliminf_{\eps\to 0}\p\{A^{\eps}\}>0$ imply
$$
\varliminf_{\eps\to 0}\p\{A^{\eps}\cap(B_1^{\eps}\cup B_2^{\eps})\}>0.
$$
Thus, we obtain
$$
\varliminf_{\eps\to 0}\p\{A^{\eps}\cap B_1^{\eps}\}>0\ \ \mbox{or}\ \
\varliminf_{\eps\to 0}\p\{A^{\eps}\cap B_2^{\eps}\}>0.
$$
It means that we can extend the interval $[u,v]$ to $[(u-\gamma)\vee 0,v]$ or $[u,(v+\gamma)\wedge
1]$, i.e.
$$
\varliminf_{\eps\to 0}\p\left\{z^{\eps}(v,t)-z^{\eps}((u-\gamma)\vee
0,t)\leq\frac{\delta}{2}\right\}>0
$$
or
$$
\varliminf_{\eps\to 0}\p\left\{z^{\eps}((v+\gamma)\wedge 1,t)-z^{\eps}(u,t)
\leq\frac{\delta}{2}\right\}>0.
$$
Noting that $\gamma$ only depends on $[\widetilde{u},\widetilde{v}]$ and applying the same
argument for new start points of the particles in finitely many steps, we obtain
$$
\varliminf_{\eps\to 0}\p\left\{z^{\eps}(v_1,t)-z^{\eps}(u_1,t)\leq\frac{\delta}{2}\right\}>0,
$$
where $(u_1,v_1)\supset [\widetilde{u},\widetilde{v}]$ and $u_1=0$ or
$v_1=1$. Next, applying the same argument for new start points of the particles, but replacing
$B_1^{\eps}\cup B_2^{\eps}$ by $B_1^{\eps}$, if $v_1=1$ or $B_2^{\eps}$, if $u_1=0$, in finitely
many steps we get
$$
\varliminf_{\eps\to 0}\p\left\{z^{\eps}(1,t)-z^{\eps}(0,t)\leq\frac{\delta}{2}\right\}>0.
$$
But it is not possible because the same argument (without $B_1^{\eps}$ and $B_2^{\eps}$) gives
$$
\lim_{\eps\to 0}\p\left\{z^{\eps}(1,t)-z^{\eps}(0,t)\leq\frac{\delta}{2}\right\}=0.
$$
The lemma is proved.
\end{proof}

\begin{proof}[Proof of Proposition~\ref{proposition_conv_in_probability}]
Using Jakubowski's tightness criterion (see Theorem~3.1~\cite{Jakubowski:1986}) and boundedness of
$\dot{\varphi}$, as in the proof of exponential tightness of $\{y^{\eps}\}$
(see Proposition~\ref{proposition_exponential_tightness}), we can prove that $\{z^{\eps}\}_{\eps\in(0,1]}$
is tight in $C([0,T],L_2(\rho))$.

Let $\{z^{\eps'}\}$ be a convergent subsequence and $z$ is its limit. By Skorokhod's theorem (see
Theorem~3.1.8~\cite{Ethier:1986}), we can define a probability space and a sequence of random
elements $\{\widetilde{z}^{\eps'}\}$, $\widetilde{z}$ on this space such that
$\law(z^{\eps'})=\law(\widetilde{z}^{\eps'})$, $\law(z)=\law(\widetilde{z})$ and
$\widetilde{z}^{\eps'}\to\widetilde{z}$ in $C([0,T],L_2(\rho))$ a.s. If we show that
$\widetilde{z}=\varphi$, we finish the proof because this implies that
$\widetilde{z}^{\eps'}\to\varphi$ in $C([0,T],L_2(\rho))$ in probability and since $\varphi$ is
non-random, $z^{\eps'}\to\varphi$ in probability. Thus, it will easily yield that $z^{\eps}\to\varphi$ in probability.

So, for convenience of notation we will assume that $z^{\eps}\to z$ a.s., instead
$\widetilde{z}^{\eps'}\to\widetilde{z}$. First we check that $z(t)\in
L_2^{\uparrow\uparrow}(\rho)$ for all $t\in[0,T]$. Let $t$ is fixed. One can show that
$$
z^{\eps}(t)\to z(t)\quad \mbox{in measure}\ \p\otimes\rho.
$$
By Lemma~4.2~\cite{Kallenberg:2002}, there exists subsequence $\{\eps'\}$ such that
$$
z^{\eps'}(t)\to z(t)\quad \p\otimes\rho-\mbox{a.e.}
$$
Set $A=\{(\omega,u):\ z^{\eps'}(u,t,\omega)\to z(u,t,\omega)\}$. Since $\p\otimes\rho(A^c)=0$, it is
easy to see that there exists the set $U\subseteq[0,1]$ such that $\rho(U^c)=0$ and $\p(A_u)=1$, for
all $u\in U$, where $A_u=\{\omega:\ (\omega,u)\in A\}$. Note, it implies that for each $u\in U$
$$
z^{\eps'}(u,t)\to z(u,t)\quad \mbox{a.s.}
$$
Let $U_{count}$ is a countable subset of $U$ which is dense in $[0,1]$. From
Lemma~\ref{lemma_coalescing_of_particles} it follows that
$$
z(u,t)<z(v,t)\quad\mbox{a.s.}
$$
for all $u,v\in U_{count}$, $u<v$.

Denote
$$
\Omega'=\bigcap_{u<v,\ u,v\in U_{count}}\{z(u,t)<z(v,t)\}.
$$
Since $U_{count}$ is countable, $\p(\Omega')=1$. Next, define
$$
\widetilde{z}(u,t,\omega)=\inf_{u\leq v,\ v\in U_{count}}z(v,t,\omega),\quad u\in[0,1],\
\omega\in\Omega'.
$$
Then for all $\omega\in\Omega'$, $\widetilde{z}(\cdot,t,\omega)\in D^{\uparrow\uparrow}$. Let us show
that $\rho\{u:\ \widetilde{z}(u,t)\neq z(u,t)\}=0$ a.s.

Denote
$$
\widetilde{\Omega}=\left(\bigcap_{u\in U_{count}}A_u\right)\cap\Omega'\cap\{z^{\eps'}(t)\to z(t)\
\mbox{in}\ L_2(\rho)\}.
$$
Then
\begin{equation}\label{f_conv_to_z}
z^{\eps}(u,t,\omega)\to z(u,t,\omega)=\widetilde{z}(u,t,\omega)
\end{equation}
for all $u\in U_{count}$ and $\omega\in\widetilde{\Omega}$. Fix $\omega\in\widetilde{\Omega}$. Since
$\widetilde{z}(\cdot,t,\omega)$ is nondecreasing, it has a countable set
$D_{\widetilde{z}(\cdot,t,\omega)}$ of discontinuous points. The countability implies that
$\rho(D_{\widetilde{z}(\cdot,t,\omega)})=0$. Take $u\in D_{\widetilde{z}(\cdot,t,\omega)}^c$, then
from monotonicity of $\widetilde{z}(\cdot,t,\omega)$ and $z^{\eps'}(\cdot,t,\omega)$, density of
$U_{count}$ and~\eqref{f_conv_to_z} we can obtain
$$
z^{\eps}(u,t,\omega)\to \widetilde{z}(u,t,\omega).
$$
Thus,
$$
z^{\eps}(\cdot,t,\omega)\to \widetilde{z}(\cdot,t,\omega)\quad \rho-\mbox{a.e.}
$$
On the other hand,
$$
z^{\eps}(\cdot,t,\omega)\to z(\cdot,t,\omega)\quad \mbox{in}\ L_2(\rho).
$$
Consequently, $z(\cdot,t,\omega)=\widetilde{z}(\cdot,t,\omega)$ $\rho$-a.e. for all
$\omega\in\widetilde{\Omega}$. So, it means that $z(t)\in L^{\uparrow\uparrow}_2(\rho)$ a.s., for
all $t\in[0,T]$.

Now we can prove that $z=\varphi$. Take $l\in C([0,1]\times[0,T],\R)$. By the dominated convergence
theorem, $\int_0^T\int_0^1l(u,t)(z^{\eps}(u,t)-u)dtdu$ converges to
$\int_0^T\int_0^1l(u,t)(z(u,t)-u)dtdu$ a.s. Integrating by parts, we get
\begin{align*}
 \int_0^T\int_0^1l(u,t)(z^{\eps}(u,t)-u)dtdu&=
\int_0^T\int_0^1L(u,t)(\pr_{z^{\eps}(t)}\dot{\varphi}(t))(u)dtdu\\
&+\int_0^T\int_0^1L(u,t)d\eta^{\eps}(u,t)du,
\end{align*}
where $L(u,t)=\int_t^Tl(u,s)ds$. The first term in the right hand side of the previous relation converges to
$\int_0^T\int_0^1L(u,t)\dot{\varphi}(u,t)dtdu$, by Lemma~\ref{lemma_semi_cont_of_F}.
The second term is the stochastic integral so we can estimate the expectation of its second moment
\begin{align*}
 \E\left(\int_0^T\int_0^1L(u,t)d\eta^{\eps}(u,t)du\right)^2&
\leq\eps\E\int_0^T\int_0^1(\pr_{z^{\eps}(t)}L(t))^2(u)dtdu\\
&\leq\eps\E\int_0^T\int_0^1L^2(u,t)dtdu\to 0,\quad \eps\to 0.
\end{align*}
Consequently, we obtain
$$
\int_0^T\int_0^1l(u,t)(z(u,t)-u)dtdu=\int_0^T\int_0^1L(u,t)\dot{\varphi}(u,t)dtdu,
$$
which easily implies $z=\varphi$. The proposition is proved.
\end{proof}

\nocite{MR2459454,MR2994690,Konarovskyi:2014:TSP,Shamov:2011}

\end{document}